\pdfoutput=1
\documentclass[a4paper]{amsart}
\usepackage{diff-style,diff-macros}

\title{Diffeomorphisms of 4-manifolds from graspers}

\author{Danica Kosanovi\'c}
\address{ETH Z\"urich, Department of Mathematics, R\"amistrasse 101, 8092 Z\"urich, Switzerland}
\curraddr{Universität Bern, Mathematisches Institut (MAI), Alpeneggstrasse 22, 3012 Bern, Switzerland}
\email{danica.kosanovic@unibe.ch}

\date{\today}

\begin{document}

\begin{abstract}
    We relate degree one grasper families of embedded circles to various constructions of diffeomorphisms found in the literature -- theta clasper classes of Watanabe, barbell diffeomorphisms of Budney and Gabai, and twin twists of Gay and Hartman.
    We use a ``parameterised surgery map'' that for a smooth 4-manifold $M$ takes loops of framed embeddings of $S^1$ in the manifold obtained by surgery on some 2-sphere in $M$, to the mapping class group of $M$.
\end{abstract}

\maketitle

\section{Introduction}\label{sec:intro}

Recently there has been remarkable progress in the study of homotopy groups of $\Diffp(\M)$, the topological group of diffeomorphisms of a smooth 4-manifold $\M$ that are the identity near (possibly empty) boundary. In particular, $\pi_0\Diffp(\M)$ is the smooth mapping class group of~$\M$.

Firstly, Watanabe~\cite{Watanabe-dim-4} constructed nontrivial classes in homotopy groups $\pi_n\Diffp(\D^4)$ for many $n\geq1$, and a potentially nontrivial \emph{theta class} $\Theta\in\pi_0\Diffp(\D^4)$. For these constructions he was inspired by \emph{clasper surgery} from Gusarov--Habiro approach to Vassiliev theory of classical knot invariants. Watanabe thus disproved the 
generalised Smale conjecture: $\Diffp(\D^d)$ is not contractible for $d=4$, even though it is for $d=1,2,3$ (by Smale~\cite{Smale} and Hatcher~\cite{Hatcher-diff}). However, the question of nontriviality of the mapping class group $\pi_0\Diffp(\D^4)$ remains open.

Secondly, Budney and Gabai~\cite{Budney-Gabai} found an infinite set of linearly independent classes in the abelian groups $\pi_0\Diffp(\D^3\times\S^1)$ and  $\pi_0\Diffp(\S^3\times\S^1)$. Moreover, they gave a general recipe for constructing diffeomorphisms of 4-manifolds, called \emph{barbell diffeomorphisms}. Another work of Watanabe~\cite{Watanabe-theta} followed, where also infinitely many elements in $\pi_0\Diffp(\D^3\times\S^1)$ are given, as well as in $\pi_0\Diffp(\Sigma\times\S^1)$, where $\Sigma$ is the Poincar\'e homology 3-sphere.

Finally, Gay~\cite{Gay} constructed an infinite list of candidate classes in $\pi_0\Diffp(\D^4)$ called \emph{Montesinos twin twists}, but together with Hartman~\cite{Gay-Hartman} they showed that this list reduces up to isotopy to at most one nontrivial element, which is 2-torsion.

Moreover, Gay~\cite{Gay} (using Cerf theory) and later Krannich--Kupers~\cite{Krannich-Kupers} (using results of Quinn and Kreck) give a general procedure for constructing classes in $\pi_0\Diffp(\M)$, which they show exhausts the whole group in the case $\M=\D^4$. Similar constructions -- which we propose to call \emph{parameterised surgery} -- have been used elsewhere, for example in~\cite{Wall-diffeos,Cappell-Shaneson,Budney-Gabai}. 

In this paper we study the following version: for any smooth 4-manifold $\M$ and a framed embedded 2-sphere $\nu S\colon\nu\S^2=\S^2\times\D^2\hra \M$, \emph{parameterised surgery of index one} is the map
\begin{equation}\label{eq-intro:ps-S}
\begin{tikzcd}
    \ps_{\nu S}\colon \pi_1(\Emb(\nu\S^1,\M_{\nu S});\nu c) \rar{\delta_{\nu c}} & \pi_0\Diffp(\M\sm\nu S) \rar{\cup\Id_{\nu S}} & \pi_0\Diffp(\M).
\end{tikzcd}
\end{equation}
Here $\M_{\nu S}\coloneqq (\M\sm\nu S)\cup_{\partial\nu S} \nu c$ is the surgery on $\nu S$, for $\nu c\cong\S^1\times\D^3$.
The map $\delta_{\nu c}$ is given by ambient isotopy extension: lift a loop of framed $\S^1\hra \M_{\nu S}$ based at $\nu c$ to a path of diffeomorphisms of $\M_{\nu S}$ (so this is an ambient isotopy extending the given isotopy of framed circles), and restrict the endpoint diffeomorphism to the complement of $\nu c$ (which it fixes, by construction). One could say $\delta_{\nu c}$ is ``circle pushing''. The map $\cup\Id_{\nu S}$ in \eqref{eq-intro:ps-S} is the extension by the identity over $\nu S$.

If $S$ is unknotted (bounds an embedded 3-ball), then $\M_{\nu S}\cong \M\#\,\S^3\times\S^1$, and we use the notation
\begin{equation}\label{eq-intro:ps}
\begin{tikzcd}
    \ps\colon \pi_1(\Emb(\nu\S^1,\M\#\,\S^3\times\S^1);\nu c) \rar & \pi_0\Diffp(\M).
\end{tikzcd}
\end{equation}
This paper explores connections between all mentioned constructions of (isotopy classes of) diffeomorphisms, using the maps $\ps_{\nu S}$ and knotted families of circles constructed using \emph{graspers} in our previous work \cite{KT-highd,K-Dax}. In Theorem~\ref{thm-intro:W-BG} we explicitly relate graspers to Watanabe's theta classes and to Budney--Gabai's barbell diffeomorphisms, in arbitrary 4-manifolds. All of them are then related to Gay's twists in Theorem~\ref{thm-intro:Gay-twist}. For $\M=\S^4$ we show in Corollary~\ref{cor-intro:final-S4} that many existing constructions of diffeomorphisms reduce to a single 2-torsion class depicted in Figure~\ref{fig-intro:Tf}. It remains open whether this class is nontrivial. The candidate diffeomorphism of $\S^4$ from the more recent work \cite[Conj.7.3]{GGHKP} is not considered here.
\begin{figure}[!htbp]
    \centering
    \includegraphics[width=0.92\linewidth]{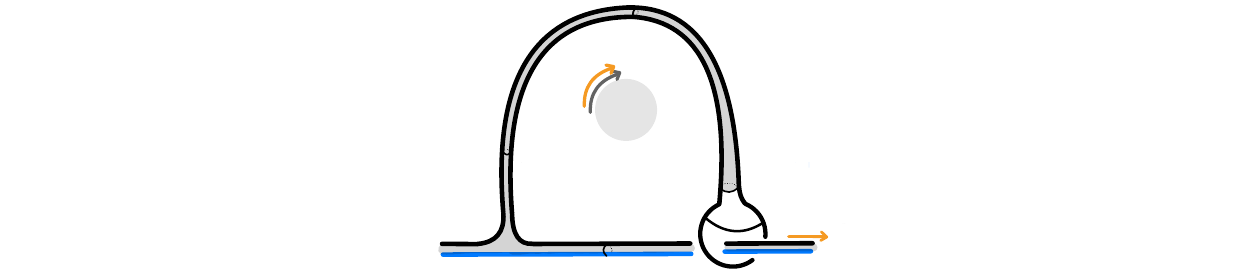}
    \caption{
        The embedded torus $T_{\prealmap(h)}\colon\S^1\times\S^1\hra \M\#\S^1\times\S^3$ is the connect-sum of a thin torus containing the blue $c=\S^1\times\{pt\}$, with a meridian sphere for $c$, along a guiding arc going around $h\in\pi_1(\M\#\S^1\times\S^3)$. The class $\ps({\prealmap(h)})\in\pi_0\Diff(\M)$ is the Gay twist $G(\Sigma_{\prealmap(h)})$ on the circle bundle $\Sigma_{\prealmap(h)}=\partial(\nu T_{\prealmap(h)})$, and then surger $c$ out. Roughly speaking, $G(\Sigma_{\prealmap(h)})$ does Dehn twists on the curves that guide the movement of $c$ around $T_{\prealmap(h)}$. The 2-torsion class $\ps({\prealmap(h)})\in\pi_0\Diff(\S^4)$ is the case $\M=\S^4$ and $h=\bc$.
    }
    \label{fig-intro:Tf}
\end{figure}

Diffeomorphisms that are in the image of $\ps$ are all pseudo-isotopic to the identity; see Proposition~\ref{prop:phc}. We briefly mention an important result of Singh~\cite{Singh}, who showed that there are infinitely many elements in $\pi_0\Diffp(\S^1\times\S^2\times[0,1])$, all pseudo-isotopic to the identity, and detected by the Hatcher--Wagoner obstruction. Earlier, such infinite collections were detected using gauge theory in $\pi_0\Diffp(\#^{2n}\mathbb{C} P^2\#^{10n+1}\ol{\mathbb{C} P^2})$ for $n\geq2$ by Ruberman~\cite{Ruberman99}.

\subsection{General 4-manifolds}
Let $\X$ be any oriented smooth 4-manifold 
and $c\colon\S^1\hra \X$ a smooth embedding, whose homotopy class we denote by $\bc\in\pi_1\X\coloneqq\pi_1(X;c(e))$. Let $\Z[\pi_1\X]$ be the free abelian group generated by the set $\pi_1\X$.
There is a group homomorphism 
\begin{equation}\label{eq-intro:r}
    \prealmap\colon{\Z[\pi_1\X]}
    \ra\pi_1(\Emb(\S^1,\X); c)
\end{equation}
that sends $h\in\pi_1\X$ to the \emph{simple grasper family with group element $h$}, shown in Figure~\ref{fig-intro:families}(i). We explain this as follows. Firstly, Figure~\ref{fig-intro:families}(ii) depicts the corresponding \emph{simple grasper}: the horizontal line union the point at infinity represents $\prealmap(h)_0=\prealmap(h)_1=c\colon\S^1\hra \X$, the two spheres are meridians $m_0,m_1$ to $c$ at two points $p_0,p_1$, and the concatenation of the bar between them followed by $c|_{[p_0,p_1]}^{-1}$ represents $h$. The grasper family $\prealmap(h)_s\colon\S^1\hra \X$ for $s\in[0,1]$ takes a piece of $c$ near $p_0$, drags it along the bar, and then twirls it around $m_1$, before going back; see Figure~\ref{fig-intro:families}(i). The orange arrow represents possible appearance of the generator $\bc$ in $h\in\pi_1\X$.
\begin{figure}[!htbp]
    \centering
    \includegraphics[width=0.88\linewidth]{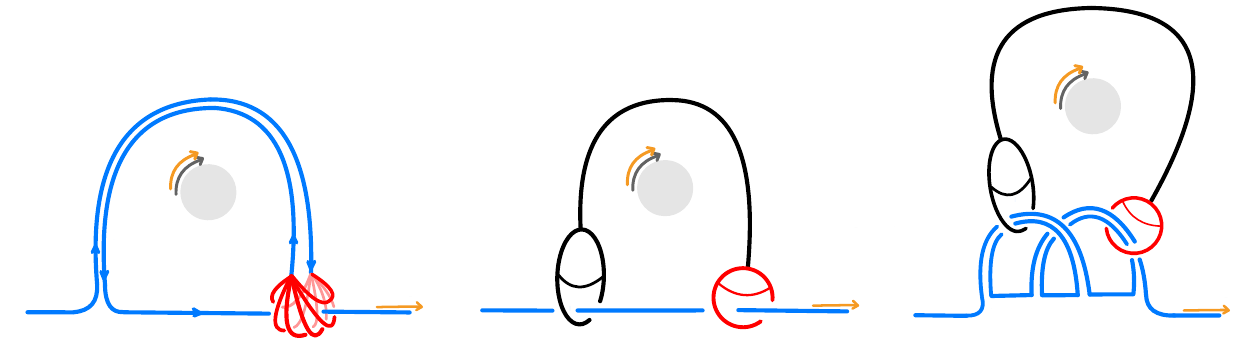}
    \caption{
        (i)~To get $\prealmap(h)_t\colon\S^1\hra X$ connect-sum the tip of the blue finger into each of the red arcs, that foliate a meridian sphere to $c$. 
        (ii)~The simple grasper for $\prealmap(h)$.
        (iii)~The semisimple grasper for $\sref^\circlearrowright_\circlearrowright(h)$.
    }
    \label{fig-intro:families}
\end{figure}

The map $\prealmap$ was defined in \cite{Gabai-disks} and in our work with Peter Teichner \cite{KT-highd}.
Following the work of Dax \cite{Dax}, we studied the kernel $\ker(\prealmap)$ in \cite{K-Dax} (see also Section~\ref{subsec:dax}). 
Continuing on this, in~\cite{K-Dax2} we study framed embeddings $\pi_1(\Emb(\nu\S^1,\X); \nu c)$, and show how the forgetful map $\Emb(\nu\S^1,\X)\to\Emb(\S^1,\X)$ behaves on $\pi_1$, depending on the second Stiefel--Whitney number $w_2(X)\in H^2(X;\Z/2)$. In particular, if $X$ is spin then this is a surjection with split kernel $\Z/2$. 


Let now $\M$ be a 4-manifold and $\nu S\colon\nu\S^2\hra \M$ a framed embedding. The first main result of this paper concerns cases when \emph{$S$ is unknotted}, so that the surgery on it gives $\M_{\nu S}\cong \M\#\S^1\tm\S^3$. This has fundamental group the free product of $\pi_1\M$ and $\Z$, whose generator we denote by $t$. We compose the maps~\eqref{eq-intro:ps} and~\eqref{eq-intro:r} into
\begin{equation}\label{eq-intro:r-ps}
\begin{tikzcd} 
    \Z[\pi_1\M\ast\Z]\rar{\prealmap} & \pi_1(\Emb(\nu\S^1,\M\#\S^1\tm\S^3);\nu c) \rar{\ps} & 
    \pi_0\Diffp(\M).
\end{tikzcd}
\end{equation} 
We recast, in terms of the map $\ps\circ\prealmap$, Watanabe's clasper classes~\cite{Watanabe-dim-4} in Section~\ref{sec:W} and Budney--Gabai barbell classes~\cite{Budney-Gabai} in Section~\ref{sec:BG}. Let us briefly recall that Watanabe's class
\[
    \wat(\Theta_{g_1,g_2})\in\pi_0\Diffp(\M)
\]
depends on an embedding $\Theta_{g_1,g_2}$ of the theta graph into $\M$. Since $\Theta_{g_1,g_2}$ is isotopic to a wedge of circles, it is determined by a pair $g_1,g_2\in\pi_1M$. 
On the other hand, Budney--Gabai's class 
\[
    \bg(\barbemb_{S_1,S_2,\W})\in\pi_0\Diffp(\M)
\]
depends on an embedding $\barbemb_{S_1,S_2,\W}$ into $\M$ of the model barbell $\barbell\coloneqq\S^2\times\D^2\natural\,\S^2\times\D^2$. This is determined by two embedded spheres $\barbemb S_1$ and $\barbemb S_2$, called the \emph{cuffs}, and an arc connecting them in their complement, called the \emph{bar}; see Figure~\ref{fig-intro:barbell}. We define $\bx,\by\in\pi_1(\M\sm\nu\barbemb(S_1\sqcup S_2))$ as the homotopy classes of the meridians of $\barbemb S_1$ and $\barbemb S_2$ respectively, and the \emph{bar word} $\W\in\pi_1(\M\sm\nu\barbemb(S_1\sqcup S_2))$ as the element determined by the bar (see Definition~\ref{def:bar-word-and-element}). When $\barbemb S_1$ is unknotted, we also define the \emph{bar group element} in $\pi_1(\M\sm\nu \barbemb S_1)\cong\pi_1\M\ast\Z$ by setting $\by=1$ and $\bx=t$ in the bar word $W$. 

\begin{mainthm}[{Theorem~\ref{thm:Theta}, Corollary~\ref{cor:Wat-implant}}]
\label{thm-intro:W-BG}
    For a smooth 4-manifold $\M$ and $g_1,g_2\in\pi_1\M$ we have
    \[
        \wat(\Theta_{g_1,g_2}) \;=\;
        \ps\circ\prealmap\big(g_1g_2^{-1}tg_2+(g_1g_2^{-1}tg_2)^{-1}\big)\;=\;
        \bg(\barbemb_{\by g_1 g_2^{-1} \bx g_2})
        \quad \in\pi_0\Diffp(\M),
    \]
    where the barbell $\barbemb_{\by g_1 g_2^{-1} \bx g_2}\colon\barbell\hra \M$ has both cuffs unknotted and the bar word 
    $\W=\by g_1 g_2^{-1} \bx g_2$.
\end{mainthm}

In fact, the proof of Theorem~\ref{thm-intro:W-BG} consists of several intermediate results.
\begin{itemize}
    \item(Section~\ref{subsec:selfref}) 
    For any $h\in\pi_1\M\ast\Z$ we define \emph{semisimple grasper} families $\sref^\circlearrowright_\circlearrowright(h)$ and $\sref^\circlearrowright(h)$ in $\pi_1(\Emb(\nu\S^1,\M\#\,\S^3\times\S^1);\nu c)$; see Figure~\ref{fig-intro:families}(iii). 
    \item(Corollary~\ref{cor:Theta}) 
    For any $g_1,g_2\in\pi_1\M$, if we denote $h=g_1g_2^{-1}tg_2$ then
        \begin{equation}\label{eq-intro:wat}
            \wat(\Theta_{g_1,g_2})=\ps\circ\sref^\circlearrowright_\circlearrowright(h).
        \end{equation} 
    \item(Proposition~\ref{prop:half-unknotted})
    For a barbell $\barbemb_{\W}$ with unknotted cuffs and bar word $\W=\prod_{i=1}^rf_i\by h_i$, for some $f_i,h_i\in\pi_1\M\ast\Z$, if we denote         $\bw_i=\prod_{j=1}^if_jh_j$ and $h= \bw_{n} \sum_{i=1}^rf_i^{-1} \bw_{i-1}^{-1}$, then
        \begin{equation}\label{eq-intro:bg}
                \bg(\barbemb_{\W})=\ps\circ \sref^\circlearrowright(h).
        \end{equation} 
    \item(Theorem~\ref{thm:selfref}) 
    For any $h\in\pi_1\M\ast\Z$ we have
        \begin{equation}\label{eq-intro:self-ref}
            \sref^\circlearrowright_\circlearrowright(h)
            =\sref^\circlearrowright(h)
            =\prealmap(h+h^{-1}).
        \end{equation} 
\end{itemize}
Knotted families similar to $\prealmap$ and $\sref^\circlearrowright$ have been defined by Gabai in~\cite{Gabai-disks}, and their relation to barbell diffeomorphisms have been also considered by Budney and Gabai, see for example~\cite[Constr.~5.25]{Budney-Gabai}. They also outline connections to Watanabe's classes~\cite[Rem.\ 5.26]{Budney-Gabai}. 

A remarkable result of Budney and Gabai is that diffeomorphisms obtained from a particular class of barbells $\delta_m\colon\barbell\hra \D^3\times\S^1$, $m\geq4$, form an \emph{infinitely generated free abelian subgroup of $\pi_0\Diffp(\D^3\times\S^1)$}, which also includes into $\pi_0\Diffp(\S^3\times\S^1)$. They use this to show that $\bg(\delta_m)(\{e\}\times\D^3)\in\pi_0\Embp(\D^3,\D^3\times\S^1)$ are linearly independent classes, in contrast to $\pi_0\Embp(\D^2,\D^2\times\S^1)\cong\pi_1\Embp(\D^1,\D^3)\cong1$ that is due to Hatcher~\cite{Hatcher-diff}.

Barbells $\delta_m$ are defined in \cite[Constr.\ 6.11]{Budney-Gabai} and depicted in Figure~\ref{fig-intro:barbell}. Here we let $g$ denote the generator of $\pi_1(\D^3\times\S^1)\cong\Z$, and define $\delta_m$ to have both cuffs unknotted and the bar word
        \[
        \W=g\by g^{m-3}\bx g^2\in\pi_1(\D^3\times\S^1\sm\nu\barbemb(S_1\sqcup S_2)).
        \]
Using~\eqref{eq-intro:bg} we can express these diffeomorphisms as follows.
\begin{cor}[{Corollary~\ref{cor:implant-D3xS1}}]
\label{cor-intro:BG}
        For the barbell $\delta_m\colon\barbell\hra \D^3\times\S^1$, $m\geq4$,  we have
    \[
        \bg(\delta_m)=
        \ps\circ\sref^\circlearrowright
        (g^{m-2}tg)\;\in \pi_0\Diffp(\D^3\times\S^1).
    \]
\end{cor}
\begin{figure}[!htbp]
    \centering
    \includegraphics[width=0.9\linewidth]{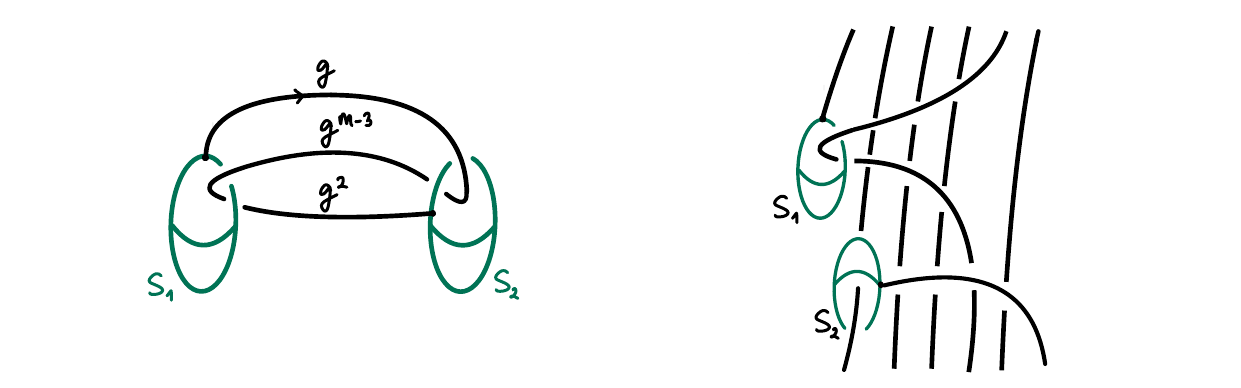}
    \caption{
        (i)~The barbell $\delta_m\coloneqq\barbemb_{g\by g^{m-3}\bx g^2}$ in $\D^3\times\S^1$ is the thickening of the two cuffs and the bar connecting them. By putting $g=1$ we view this as the barbell $\barbemb_{\by\bx}$ in $\D^4\subset\S^4$.
        (ii)~Another picture of $\delta_m\subset\D^3\times\S^1$, with $g\in\Z\cong\pi_1(\D^3\times\S^1)$ as the $y$-axis.
    }
    \label{fig-intro:barbell}
\end{figure}
By a different technique, Watanabe~\cite{Watanabe-theta} shows that his theta classes $\wat(\Theta_{g^{p-1},g})$, $p-1\geq3$, form an infinitely generated free abelian subgroup of $\pi_0\Diffp(\Sigma_{2,3,5}\times\S^1)$, for the Poincar\'e homology 3-sphere $\Sigma_{2,3,5}$. He points out that these classes do not come from $\pi_1\Emb(\nu\S^1,\Sigma_{2,3,5}\times\S^1)$. However,~\eqref{eq-intro:wat} implies that they do lie in the image of the map $\ps$ for $M\coloneqq\Sigma_{2,3,5}\times\S^1$:
\begin{cor}
\label{cor-intro:wat}
    For Watanabe's theta diffeomorphisms of $\Sigma_{2,3,5}\times\S^1$ we have
    \[
        \wat(\Theta_{g^{p-1},g})=\ps\circ\sref^\circlearrowright_\circlearrowright(g^{p-2}tg),
    \]
    where $g$ is represented by $\{pt\}\times\S^1\subset\Sigma_{2,3,5}\times\S^1$ and $t$ by $\{pt\}\times\S^1\subset\S^3\times\S^1$.
\end{cor}

Since the classes $\wat(\Theta_{g^{p-1},g})$ are supported in $\D^3\times\S^1\subset\Sigma_{2,3,5}\times\S^1$, they must be nontrivial and distinct in $\pi_0\Diffp(\D^3\times\S^1)$ as well. This raises the question whether this infinite list differs from the list $\{\bg(\delta_m)\}_{m\geq4}$ of Budney and Gabai. By putting $m=p$ in the last two corollaries and using~\eqref{eq-intro:self-ref} we see that these lists in fact agree.
\begin{cor}
    Watanabe's and Budney--Gabai's countable lists of linearly independent classes in $\pi_0\Diffp(\D^3\times\S^1)$ are exactly the same.
\end{cor}

In Section~\ref{sec:BG} we study more general barbells, and for a 4-manifold $\M$ and any framed embedded 2-sphere $\nu S$ the maps
\begin{equation}\label{eq-def:ps-prealmap}
\begin{tikzcd}
    \Z[\pi_1\M_{\nu S}]\rar{\prealmap} & 
    \pi_1(\Emb(\nu\S^1,\M_{\nu S});\nu c) \rar{\ps_{\nu S}} & 
    \pi_0\Diffp(\M).
\end{tikzcd}
\end{equation}
We show the following.
\begin{cor}[{Corollary~\ref{cor:implant-formula}}]\label{cor-intro:nullhtpic-barbell}
    Let $\M$ be any smooth 4-manifold and a barbell  $\barbemb\colon\barbell\hra \M$ with one of the cuffs, say $\barbemb S_2$, nullhomotopic in $\M$. Then the resulting barbell diffeomorphism of $\M$ is of the form $\ps_{\nu \barbemb S_1}\circ\prealmap(h)$ for some $h\in\Z[\pi_1\M_{\nu\barbemb S_1}]$.
\end{cor}
In fact, in Theorem~\ref{thm:BG} we determine the element $h$ in terms of $\W$ and $\barbemb S_2$.
For a more explicit formula in the case when one cuff is unknotted see Proposition~\ref{prop:half-unknotted}, and for examples when both cuffs are unknotted see Section~\ref{subsec:barbells-unknotted}.

\subsection{Parameterised surgery on a simple grasper is a Gay twist}\label{subsec-intro:ps-Gay}

Using the work of David Gay~\cite{Gay} we can describe diffeomorphisms in the image of $\ps_{\nu S}\circ\prealmap$ more explicitly.

\begin{defn}\label{def:Gay-twist-3T}
    For a 4-manifold $\X$ and a framed embedded 3-torus $\nu\Sigma\colon\S^1_a\times\S^1_b\times\S^1_\theta\times[0,1]\hra \X$ we define the \emph{Gay twist $G(\Sigma)\in\pi_0\Diffp(\X)$ along $\Sigma$} to be the identity on $\X\sm\nu\Sigma$, and the product of the identity on $\S^1_b\times\S^1_\theta$ with the positive Dehn twist on $\S^1_a\times[0,1]$.
\end{defn}

The parametrisation of the 3-torus $\Sigma$ is of crucial importance in the definition of $G(\Sigma)$ (the factor $[0,1]$ can be recovered from the orientations). The key class of examples arises as follows.

\begin{defn}\label{def:Gay-twist-2T}
    If $\nu T\colon\S^1_a\times\S^1_b\times\D^2\hra \X$ is a framed embedding of the 2-torus, consider the 3-torus $\Sigma_T\coloneqq\partial(\nu T)\cong T\times\S^1_\theta$ obtained as the normal circle bundle of $T$. Let $\nu\Sigma_T\colon\S^1_a\times\S^1_b\times\S^1_\theta\times[0,1]\hra \X$ be its tubular neighbourhood. We define the \emph{Gay twist $G(\nu T)\in\pi_0\Diffp(\X)$ associated to $T$} as the Gay twist along $\nu\Sigma_T$.

    Since $G(\nu T)$ is the identity on a small neighbourhood $\nu c$ of $c=T(\{pt\}\times\S^1_b)\subset \X\sm\nu\Sigma_T$, we can define the diffeomorphism of the surgered manifold $\X_{\nu c}=(\X\sm\nu c)\cup_{\partial \nu c}(\D^2\times\S^2)$, by removing $\nu c$ and extending by the identity over $\D^2\times\S^2$. We call this the \emph{surgered Gay twist} and write
    \[
    \ol{G}(\nu T)\in\pi_0\Diffp(\X_{\nu c}).\qedhere
    \]
\end{defn}

Observe that given an embedded framed 2-torus $\nu T_f\colon \S^1_a\times\S^1_b\times\D^2\hra \X$ we can foliate it by $f_a=T_f|_{\{a\}\times\S^1_b}$ for $a\in\S^1$ to obtain a family $f\in\pi_1(\Emb(\S^1,\X); c)$, where $c=T_f|_{\{pt\}\tm\S^1_b}$. The framing on $f$ is recovered by adding to the normal framing of $f_a$ in $T_f$ the normal framing $\nu T_f$ of $T_f\subset \X$. Note that if $\X= \M_{\nu S}$ then $\X_{\nu c}$ is diffeomorphic to $\M$. The following is essentially a result due to Gay~\cite[Thm.\ 4]{Gay} and Gay and Hartman~\cite[Lem.\ 6]{Gay-Hartman}, rewritten using our notions of grasper classes and parameterised surgery. 

\begin{mainthm}\label{thm-intro:Gay-twist}
    For any $\nu f\in\pi_1(\Emb(\nu\S^1,\M_{\nu S}); \nu c)$ that is obtained by foliating a framed embedded torus containing $\nu c$ by framed circles, the class $\ps_{\nu S}(\nu f)\in\pi_0\Diff(\M)$ is represented by a surgered Gay twist. More precisely, there exists $\nu T_f\colon\nu(\S^1\times\S^1)\hra \M_{\nu S}$
    such that 
    \[
        \ps_{\nu S}\big(\prealmap(\nu f)\big)=\ol{G}(\Sigma_{T_f}).
    \]
\end{mainthm}
\begin{proof}
    Using the coordinates on $\nu T_f$ as in \cite[Thm.\ 4]{Gay}, we can write down an explicit ambient isotopy extension $F_t\in\Diffp(\M_{\nu S})$ of $\nu f_t$, which is clearly supported on $\nu T_f$. Let $(r,\theta)\in\D^2$ be radial coordinates, and $\vartheta\colon[0,1]\to[0,1]$ a smooth non-increasing function such that $\vartheta=1$ near $0$ and $\vartheta=0$ near $1$. Then define
    \[
        F_t(a,b,r,\theta) = (a+ t\vartheta(r), b, r, \theta).
    \]
    In words, for a fixed $(b,r,\theta)$ the points on the circle $\S^1_a\times\{(b,r,\theta)\}$ shift by a function of $r$. Points near $T_f$ and $\partial(\nu T_f)$ do not move, and the remaining space can be viewed as $\nu\Sigma_f$ for $\Sigma_f=\partial(\nu T_f)=\S^1_a\times\S^1_b\times\{1/2\}\times\S^1_\theta$, on which the formula performs the Gay twist.
    Note that $c$ corresponds to $a=0$ and $(r,\theta)=(0,0)$ so we have $F_t\circ c(b)=(t,b,0,0)$, which is precisely the curve $f_t$. Since $F_0=\Id$, we have that $F_t$ is desired ambient isotopy extension.
    
    Finally, $\delta_{\nu c}(f)=F_1|_{\M\sm\nu S}$ by definition, and since $\Sigma_f\subset \M\sm\nu S$, this is still given as $G(\nu \Sigma_f)$. To obtain $\ps\circ\prealmap(f)$ we extend it by the identity on $\nu S\cong\D^2\times\S^2$, which is $\ol{G}(\nu \Sigma_f)$ by definition.
\end{proof}
\begin{rem}
    In Remark~\ref{rem:emb-torus} we will see that each class $f\in\pi_1(\Emb(\nu\S^1,\M_{\nu S}); \nu c)$ that is in the image of $\prealmap$ is represented by a foliation of a framed embedded torus, as in Figure~\ref{fig-intro:Tf}. 
    The group $\pi_1(\Emb(\nu\S^1,\M_{\nu S});\nu c)$ is in general a nontrivial extension of $\im(\prealmap)\cong\Z[\pi_1\M_{\nu S}]/\ker(\prealmap)$ as in~\eqref{eq:Dax-extension}, and not every class is represented by an embedded 2-torus.
\end{rem}

\begin{rem}\label{rem:3-steps}
    From the definition of the surgered Gay twist, we see that one can approach the question of its nontriviality in two steps: we first detect nontriviality of a Gay twist, and then whether it survives surgery (no pun intended).
    This is analogous to the factorisation of $\ps_{\nu S}$ in \eqref{eq-intro:ps-S}. In fact, we can proceed in three steps: compute the group $\pi_1(\Emb(\nu\S^1,\M_{\nu S});\nu c)$, then the kernel $\ker(\delta_{\nu c})\subseteq\pi_1(\Emb(\nu\S^1,\M_{\nu S});\nu c)$, and then study the map $\cup\Id_{\nu S}$ on the image of $\delta_{\nu c}$.
\end{rem}

\subsection{The 4-sphere}

For $\M=\S^4$ by \cite{Budney-Gabai,K-Dax} we have $\Z[\pi_1(\M\#\,\S^3\times\S^1)]=\Z[t,t^{-1}]$ and $\ker(\prealmap)=\langle t^{-k}+t^{k-1},1\rangle$, so the quotient of $\Z[t,t^{-1}]$ by $\ker(\prealmap)$ is isomorphic to $\Z^\infty\coloneqq\Z\langle t,t^2,\dots\rangle$. The cokernel of $\prealmap$ is $\Z$, which splits back as the loop of circles $\rot_c$ given by rotating the source $\S^1$ once. Moreover, since $\S^4$ is spin by the mentioned result of~\cite{K-Dax2} the framed space has an additional $\Z/2$, generated by the loop of framed circles $\rot_c$ given by rotating the normal disks. Therefore, we have
\begin{equation}\label{eq:case-S3xS1}
\begin{tikzcd}
    \Z\times\Z/2\times \Z^\infty\ar{rrr}{\rot_c\times \rot_{\nu c}\times \prealmap}[swap]{\cong} &&&
    \pi_1(\Emb(\nu\S^1,\S^3\times\S^1);\nu c)\rar{\ps} &
    \pi_0\Diff(\S^4).
\end{tikzcd}
\end{equation}
It is immediate that $\ps(\rot_c(1))=0$ -- undo it by an ambient isotopy extending the rotation of the bounding disk for $c$ in $\S^4$, as well as $\ps(\rot_{\nu c}(1))=0$ -- rotate the 3-balls in $\S^1\tm\D^3=\S^3\sm\nu c$ back. Thus, it remains to study $\ps(\prealmap(t^i))$ for $i\geq1$.

The relations $\ker(\prealmap)=\langle t^{-k}+t^{k-1},1\rangle$ imply $\prealmap(t^{-1})=\prealmap(1)=0$, so we have $\wat(\Theta)=\bg(\barbemb_{\by\bx})=\ps\circ\prealmap(t)$ by Theorem~\ref{thm-intro:W-BG}. Moreover, in Corollary~\ref{cor:implant-D4} we will express any barbell diffeomorphism in $\S^4$ with unknotted cuffs in terms of $\ps\circ\prealmap$, so that also $\bg(\barbemb_{\by^{-1}\bx})=\ps\circ\prealmap(t)$ for example.

\begin{cor}
    In $\pi_0\Diff(\S^4)$ there are equalities
    \[
\wat(\Theta)=\bg(\barbemb_{\by\bx})=\bg(\barbemb_{\by^{-1}\bx})=W(1)=\ps\circ\prealmap(t).
    \]
\end{cor}

Here $W(1)$ stands for a Montesinos twin twist of Gay~\cite{Gay}. Namely, the Montesinos twin twist $\tau_{(R,S)}$ from \cite[Thm.\ 4]{Gay} and \cite[Lem.\ 6]{Gay-Hartman} is by definition the Gay twist $\ol{G}(\nu T_f)$, where $R$ is the embedded sphere obtained by surgery on $T_f$ along $c$ and $S$ is the meridian sphere of $c$. Moreover, we saw in Theorem~\ref{thm-intro:Gay-twist} that $\ol{G}(\nu T_f)$ is precisely $\ps\circ\prealmap(f)$. 

For $\M=\S^4$ Gay constructs Montesinos twins $(R(i),S)$, $i\geq 0$, using a ``snake whose tail passes through his head after linking $S$ $i$ times'', and defines $W(i)\coloneqq\tau_{(R(i),S)}$. By comparing \cite[Fig.~8]{Gay} to our Figure~\ref{fig:sref}(ii) for all $i\geq1$ we more precisely have $W(i)=\ps\circ \sref^\circlearrowright(t^i)$.

A remarkable result of Gay and Hartman~\cite[Lem.\ 10,Fig.\ 2]{Gay-Hartman} says $W(i)^{-1}=W(1)^i$,
and in particular $W(1)^2=\Id$. A related observation was made by Budney and Gabai~\cite[Prop.5.17]{Budney-Gabai}. In our language (see Lemma~\ref{lem:2-torsion}):
\[
    \ps\circ\prealmap(-t^i)=\ps\circ\prealmap(it)=(\ps\circ\prealmap(t))^i.
\]
Among the experts it was expected that Watanabe's theta class $\wat(\Theta)$ is also 2-torsion (since the theta graph is in its graph complex), as confirmed by the above corollary. Gay mentions a relation of twin twists to barbells~\cite[4]{Gay}. 
\begin{rem}
    David Gay has recently announced a proof of the equality $\wat(\Theta)=W(1)$.
\end{rem}

Finally, let us point out that for an arbitrary barbell $\barbemb\colon\barbell\hra \S^4$ the cuff $\barbemb S_2$ is nullhomotopic in $\S^4$, so satisfies conditions of Corollary~\ref{cor-intro:nullhtpic-barbell}. However, we point out that it is not clear a priori if these diffeomorphisms lie in the image of $\ps$ (that is, the 1-2 subgroup of \cite{Gay}).

\begin{maincor}\label{cor-intro:final-S4}
    For $\M=\S^4$ the image of the parameterised surgery map $\ps$, as well as any barbell diffeomorphism with one unknotted cuff, consists of at most of one class,
    \[
    \ps\circ\prealmap(t)=\wat(\Theta)=\bg(\barbemb_{\by\bx})=W(1),
    \]
    and it is 2-torsion. This class can also be described as the surgered Gay twist $\ol{G}(\nu T)$ on the 3-torus that is the normal circle bundle of $T\colon\S^1\times\S^1\hra\S^1\times\S^3$ from Figure~\ref{fig-intro:Tf}.
    More generally, any barbell diffeomorphism $\bg(\barbemb)$ in $\S^4$ can be expressed as $\ps_{\nu\barbemb S_1}(f)$ for some $f\in\pi_1(\Emb(\nu\S^1,\S^4_{\nu\barbemb S_1});\nu c)$.
\end{maincor}

\subsection*{Conventions}
All manifolds and embeddings are smooth and oriented.
An embedding is denoted by $K\colon Y\hra \X$, and
$\nu K\colon\nu Y\hra \X$ means a parameterised tubular neighbourhood of $K$. If $Y=\S^{k-1}$ for some $k\geq1$ then $\X_{\nu K}\coloneqq (\X\sm\nu K)\cup\D^k\times\S^{d-k}$ is the result of surgery on $\nu K$.

By a \emph{meridian} $m$ of $S\colon\S^2\hra \X$, we mean the boundary of a small normal disk at a point. This disk has positively orientation if that of $S$ followed by that of the disk gives the chosen orientation of $\X$. We orient $m$ using the ``outward normal first''.
In our drawings the fourth dimension is oriented towards the reader (of relevance when describing movement of embedded arcs with time).

For a group $\pi$ the free abelian group generated by $\pi$ is denoted by $\Z[\pi]$ and an element is written as $\sum\e_i h_i$ for $\e_i\in\{-1,+1\}$ and $h_i\in\pi$. There is an involution $\ol{\sum\e_i h_i}\coloneqq \sum \e_i h_i^{-1}$.

For a 4-manifold $\X$ the intersection pairing $\lambda\colon\pi_3\X\times\pi_1\X\to\Z[\pi_1\X]$ is the equivariant intersection number between a 3-sphere $a=[A]\in\pi_3\X$ and a circle $h=[\gamma]\in\pi_1\X$. It is computed by picking respective representatives $A\colon\S^3\to \X$ and $\gamma\colon\S^1\to \X$ which intersect transversely, in finitely many points $x_i\in \X$, for $i=1,\dots,n$, and letting $\lambda(a,h)=\sum_{i=1}^n\e_ih_{x_i}$. 
The double point loop $h_{x_i}$ follows a path from the basepoint to $x_i$ on $A$ and then back on $\gamma$, and the sign is positive if and only if $dA(\R^3)\oplus d\gamma(\R)$ gives a positive basis of $d\X|_{x_i}$.
We also use the reduced intersection pairing $\lambdabar\colon\pi_3\X\times\pi_1\X\to\Z[\pi_1\X\sm1]$ that sends $(a,h)$ to $\lambda(a,h)$ minus the coefficient at $1\in\pi_1\X$.

We remark that in \cite{Budney-Gabai} the \emph{model barbell} $\barbell$ is called the barbell, a \emph{barbell} $\barbemb\colon\barbell\hra \M$ is a barbell implantation, whereas \emph{twirling} is spinning. In \cite{KT-4dLBT,K-Dax} twirling is called swinging. 

\subsection*{Acknowledgements}
I wish to thank Pete Teichner for many discussions about Watanabe's classes, and the dotted sphere notation that inspired Proposition~\ref{prop:phc}. Many thanks to Peter Feller and Oscar Randal-Williams for the help with a lemma that unfortunately did not make it to the paper (but perhaps will to some future one!). I am immensely grateful to Daniel Hartman for pointing out a mistake in a draft version. Thank you to Tadayuki Watanabe for useful comments.

\begin{spacing}{0.1}
    \tableofcontents
\end{spacing}

\section{Knotted families from simple graspers}\label{sec:graspers}

Throughout this section we fix an oriented compact smooth 4-manifold $\X$ with (possibly empty) boundary, and a smooth embedding $c\colon\S^1\hra \X$, our basepoint for $\Emb(\S^1,\X)$, with homotopy class $\bc\in\pi_1\X=\pi_1(X,c(e))$.
In Section~\ref{subsec:graspers} we recall the definition of degree one grasper family:
\begin{equation}\label{eq-def:realmap}              
    \prealmap\coloneqq\prealmap_{\X,c}\colon\Z[\pi_1\X]\ra\pi_1(\Emb(\S^1,\X); c).
\end{equation}
In Section~\ref{subsec:Dax-invt} we study the Dax invariant, which is an inverse of $\prealmap$ on its image. In Section~\ref{subsec:links} we briefly discuss some families of links, that will be needed in Watanabe's construction.

\subsection{Simple grasper families}
\label{subsec:graspers}

We assume that $\S^1$ is oriented and has a fixed basepoint $e\in\S^1$. In our pictures the neighbourhood of $c(e)$ is not drawn, and the rest of $c$ is represented by the $x$-axis. The small arrow next to the $x$-axis indicates the homotopy class $\bc$ of $c$, and the double arrow around the disk in the middle of the pictures indicates a group element $h\in\pi_1\X$.

When $\X= \M\#\,\S^3\times\S^1$ and $c=\{pt\}\times\S^1$, it represents the generator $\bc=t$ of the free factor $\Z<\pi_1(\M\#\,\S^3\times\S^1)$. Thus, the double arrow indicates a word $h\in\pi_1\M\ast\Z\cong\pi_1(\M\#\,\S^3\times\S^1)$.

We now define the group homomorphism \eqref{eq-def:realmap}, following~\cite{Gabai-disks,KT-highd}. 

\begin{defn}\label{def:realmap}
    For a group element $h\in\pi_1(\X)$ we define a family of embeddings
    \[
        \prealmap(h)_s\colon\S^1\hra \X,
    \]
    where $s\in[0,1]$ and $\prealmap(h)_0=\prealmap(h)_1=c$ as follows. We first pick an embedding $\TG\colon\D^4\hra \X$ which intersects $c$ in exactly two subintervals, that are neighbourhoods of two points $p_0,p_1$, so that $c(e)<p_0<p_1$ and the loop that goes from $p_0$ to $p_1$ along $\TG$ and then back to $c(e)$ along oppositely oriented $c$, represents the given element $h$.

    Then $\prealmap(h)_s$ is the isotopy that first takes a neighbourhood of $p_0$ and drags it within $\TG$ until it reaches the meridian sphere of $c$ at $p_1$, then twirls around this sphere, and then goes back the same way. In the twirling motion we use a fixed foliation of $\S^2$ by arcs that have endpoints fixed.

    We represent this family as in Figure~\ref{fig-intro:families}(ii):  we draw only the meridian spheres at $p_0$ (called the \emph{root}) and $p_1$ (called the \emph{leaf}) and an arc in $\TG$ connecting them. \emph{We call this object a simple grasper.}

    Next, we define $\prealmap(-h)$ by the same picture as for $\prealmap(h)$ except that the leaf sphere is the negatively oriented meridian (i.e.\ the linking number with $c$ is $-1$), see Figure~\ref{fig:realmap-ext}(i).

    Finally, given $\sum_{i=1}^m\epsilon_i h_i\in\Z[\pi_1\X]$ we pick disjoint balls $\TG_i$ with associated group elements $h_i$ and $p_0^m<\dots<p_0^1<p_1^1<\dots<p_1^m$, see Figure~\ref{fig:realmap-ext}(ii) for $m=2$. We define the family $\prealmap(\sum_{i=1}^m\epsilon_i h_i)$ by letting the corresponding families $\prealmap(\epsilon_i h_i)$ in $\TG_i$ run simultaneously in time parameter $s\in[0,1]$.
\end{defn}

\begin{figure}[!htbp]
    \centering
    \includegraphics[width=0.9\linewidth]{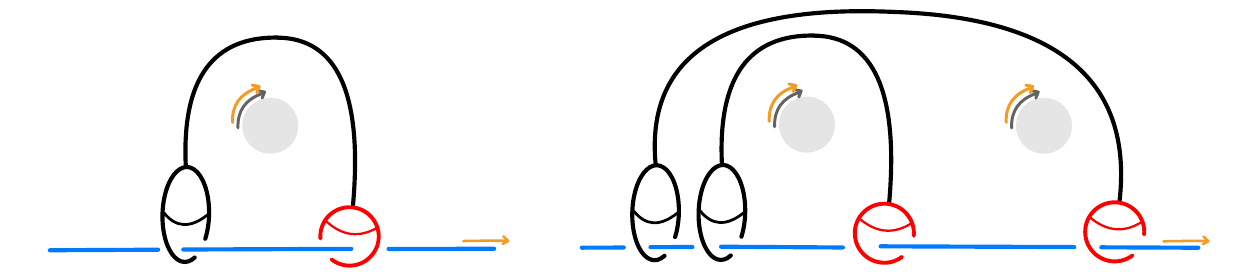}
    \caption{
        (i)~The family $\prealmap(-h)$.
        (ii)~The family $\prealmap(h_1+h_2)$.
    }
    \label{fig:realmap-ext}
\end{figure}

\begin{rem}\label{rem:half-twist}
    One can show that $\prealmap(\sum\epsilon_i h_i)=\prod\prealmap(h_i)^{\epsilon_i}$, so $\prealmap$ is a homomorphism. See the proof of \cite[Thm.\ 4.21]{KT-highd}. Note that the product on the right is in $\pi_1$, so in the $s$-direction, whereas our definition of the left hand side uses a sort of a product in the circle direction (use all graspers in parallel). See also Lemma~\ref{lem:inverse} below for an alternative descriptions of $\prealmap(-h)=\prealmap(h)^{-1}$. 
\end{rem}

\begin{rem}\label{rem:emb-torus}
    It is not hard to see that each $\prealmap(\pm h)$ has a representative given by foliating an embedded 2-torus: start by  a small rotation of $c$ in the normal direction giving a thin embedded torus; then attach to this a tube that goes around $h$ and grabs the thin torus at a big enough meridian of $c$. Moreover, $\prealmap(f)$ for any $f=\sum\epsilon_i h_i$ is a thin torus plus several such tubes. The previous remark says that it does not matter if we foliate this by going over the tubes one after another, or if we foliate them in parallel.
\end{rem}


\subsection{The Dax invariant}
\label{subsec:Dax-invt}
    
    The image of $\prealmap$ consists of (nonbased) nullhomotopic families. That is, there are group extensions
    \begin{equation}\label{eq:Dax-extension}
        \begin{tikzcd}
        &&  \faktor{\pi_2\X}{b=\bc\cdot b}\dar[tail]{} 
        \\
        \faktor{\Z[\pi_1\X]}{\ker(\prealmap)}\rar[tail]{\prealmap_{\X}} 
        & \pi_1(\Emb(\S^1,\X); c)\rar[two heads]{\iota} 
        & \pi_1(\Imm(\S^1,\X); c)\dar[two heads]{\ev_e}\\
        && \Fix_{\bc}(\pi_1\X),
        \end{tikzcd}
    \end{equation}
    where $\Fix_{\bc}(\pi_1\X)\coloneqq\{h\in\pi_1\X: h=\bc h\bc^{-1}\}$ is the centraliser of $\bc\in\pi_1\X$.
    These sequences do not split in general, and $\pi_1(\Emb(\S^1,\X); c)$ is often nonabelian, see \cite{KT-4dLBT,K-Dax,K-Dax2}. For some information on the kernel of $\prealmap$ see Section~\ref{subsec:dax} below.
    
On the image of $\prealmap$ (equivalently, on the kernel of the inclusion map $\iota$ to immersions from~\eqref{eq:Dax-extension}) one can define an explicit inverse, the Dax invariant
    \[\begin{tikzcd}
        \Da\colon\; \im(\prealmap)=\ker(\iota)
        \rar[tail]{} &
        \faktor{\Z[\pi_1\X]}{\ker(\prealmap)}
    \end{tikzcd}
    \]
that counts double points (with signs and group elements) in any homotopy. More precisely, let us give an algorithm that computes $\Da(f)$ for $f\in\pi_1(\Emb(\S^1,\X); c)$ such that $\iota(f)=0$.

\begin{alg}\label{alg}\hfill
\begin{enumerate}
    \item Pick a homotopy $F\colon[0,1]\to\Omega\Imm(\S^1,\X)$, $\tau\mapsto F_{-,\tau}$, from $F_{-,0}=f$ to $F_{-,1}=\const_c$ such that the circle $F_{s,\tau}\colon\S^1\imra \X$ is not embedded only at finitely many times $\tau=\tau_j\in[0,1]$ and $s=s_j\in[0,1]$. Note that $F_{0,\tau}=F_{1,\tau}=c$ for all $\tau\in[0,1]$.
    \item For every $j$ make an ordered list $c(e)<x^j_1<\dots<x^j_{k_j}$ of double points $x^j_i\in \X$ of the nonemebedded circle $F_{s_j,\tau_j}(\S^1)$, using the positive orientation of $\S^1$ and the basepoint $e$ for the ordering. Note that at each $x_i^j$ there are two sheets (local pieces of the circle $F_{s_j,\tau_j}$ passing through $x_i^j$), and we can order them as well.
    \item Modify $F$ so that for each $x_i^j$: (a) only one of the sheets moves with $\tau$ or $s$, whereas the other sheet stands still in both of these directions, (b) there is a chart $\R\times\R^3\subset \X$ around $x_i^j$, so that the $s$-derivative of the moving sheet is parallel to the positive $\R\times\{\Vec{0}\}$ direction (i.e.\ moves from past into future with $s$).
    \item For each $x_i^j$ write down the double point \textbf{loop} $h_i^j\in\pi_1\X$: start at $c(e)=F_{s_j,\tau_j}(e)$ and go on $F_{s_j,\tau_j}$ until you reach $x_i^j$ for the \emph{first time}; then change to the  \emph{other sheet} and run on it in the \emph{reversed orientation} until you reach $c(e)$ again.
    \item For each $x_i^j$ write down the double point \textbf{sign} $\epsilon(x_i^j)=\epsilon_1\cdot\epsilon_2\in\{\pm1\}$ as follows: $\epsilon_1=1$ (resp.\ $\epsilon_1=-1$) if the first (resp.\ second) sheet is moving, and $\epsilon_2$ is positive if and only if the following basis of $\{0\}\times\R^3$ is positively oriented: the first vector is the derivative in the $\tau$-direction of the moving sheet, the second vector is tangent to the moving sheet, the third vector is tangent to the non-moving sheet.
    \item Finally, let $\Da(f)\coloneqq\sum_j\sum_i\epsilon(x_i^j)h_i^j$.
\end{enumerate}
\end{alg}
\begin{proof}[Proof of the Algorithm]
    Let us show that this is a valid way to compute $\Da(f)$, using its definition from \cite{KT-highd} and assuming the main theorem there. In other words, we know that $\ker(\iota)=\im(\prealmap)$ and that there is an invariant $\Da$, defined in \cite[Sec.\ 4.1.3]{KT-highd}, that is well defined, i.e.\ does not depend on the choice of a homotopy $F$. 

    Let us first show that the Algorithm gives
    \[
        \Da(\prealmap(h))=h.
    \]
    First note that in our parametrisation of the family $\prealmap(h)$ (Figure~\ref{fig:realmap-Dax}(i)) there is a single circle $\prealmap(h)_{s_1}$ contained completely in the present slice (Figure~\ref{fig:realmap-Dax}(ii)). A homotopy from the family $\prealmap(h)$ to the constant family is given by pulling apart the linking of the red sphere with the horizontal line, using the meridian ball filling that sphere in. During the homotopy, a single double point $x_1$ occurs, precisely in the homotopy $\prealmap(h)_{s_1,\tau}$ of the circle $\prealmap(h)_{s_1}$ (Figure~\ref{fig:realmap-Dax}(iii)). It is obtained by pushing the undercrossing arc up, that is, the $\tau$-derivative points towards the reader. To compute the associated loop note that the first sheet is on the tip of the finger (the vertical red arc in Figure~\ref{fig:realmap-Dax}(iii)). By the time we reach the double point we have already traversed $h$, and then we go back by traversing the whole finger, which gives $hh^{-1}=1$, so in total we have $h$. For the sign, note that the first sheet (red) is the one to move (double arrow). Thus, the basis ``($\tau$-derivative, moving, nonmoving)'' is positive, so $\epsilon(x_1)=(+1)(+1)=+1$. Therefore, $\Da(\prealmap(h))=h$ as claimed. 
    
\begin{figure}[!htbp]
    \centering
    \includegraphics[width=0.95\linewidth]{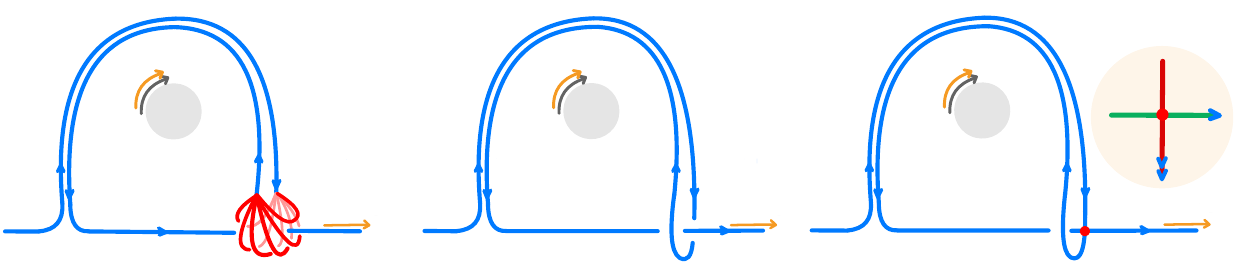}
    \caption{
        (i)~The family $\prealmap(h)$. 
        (ii)~The circle in this family completely contained in the present, $\{0\}\times\R^3$. 
        (ii)~The only immersed circle in the family $\prealmap(h)_{s,\tau}$, and the local picture around the double point.
    }
    \label{fig:realmap-Dax}
\end{figure}
    For an element $f\in\ker(\iota)=\im(\prealmap)$ we know that $f=\prealmap(\sum_i\epsilon_ih_i)=\prod_i\prealmap(h_i)^{\epsilon_i}$, and a nullhomotopy $F$ for $f$ with given properties exists: it can be built by running in parallel the nullhomotopies for each $\prealmap(h_i)$. The previous paragraph implies that the Algorithm gives the correct $\Da$ invariant. 
    
    Finally, we have to show that if $F$ is another nullhomotopy of $f$ with given properties, the Algorithm has the same output. The loops are defined in the same way as for $\Da$, so we need only check the signs. In the definition of $\Da$ they arise by comparing the standard basis $(e_s,e_\tau,w,x,y,z)$ of $\I\times\I\times\R^4$, where $\R^4$ is the chart in $\X$ around the double point $x_i^j$, to the basis obtained by taking the derivatives of $\wt{F}\colon\I\times\I\times\S^1\to \I\times\I\times \X$ at the two sheets:
    \[
        \left(\frac{\partial \wt{F}_1}{\partial s},\frac{\partial \wt{F}_1}{\partial \tau},\frac{\partial \wt{F}_1}{\partial \theta},
        \frac{\partial \wt{F}_2}{\partial s},\frac{\partial \wt{F}_2}{\partial \tau},\frac{\partial \wt{F}_2}{\partial \theta}\right).
    \]
    Let $F_m$ stand for the restriction of $F$ to the moving sheet, and $F_n$ to the non-moving. Note that our choice of $F$ is such that $\partial\wt{F}_n/\partial s=e_s$ and $\partial\wt{F}_n/\partial \tau=e_\tau$ and  $\partial\wt{F}_m/\partial s=e_s+w$, see the step (3) of the Algorithm. Then to find $\epsilon_2$ the step (5) of the Algorithm compares the standard basis $(w,x,y,z)$ of $\R\times\R^3\cong T\X|_{x_i^j}$ to the following:
        \begin{equation}\label{eq:alg-sign}
            \left(w,\frac{\partial F_m}{\partial \tau},\frac{\partial F_m}{\partial \theta},\frac{\partial F_n}{\partial \theta}\right).
        \end{equation}
    For $F_1=F_m$ the above basis $(e_s+w,\frac{\partial F_m}{\partial \tau},\frac{\partial F_m}{\partial \theta},e_s,e_\tau,\frac{\partial F_n}{\partial \theta})$ is positive if and only if the basis $(w,\frac{\partial F_m}{\partial \tau},\frac{\partial F_m}{\partial \theta},e_s,e_\tau,\frac{\partial F_n}{\partial \theta})$ is if and only if  $(e_s,e_\tau,w,\frac{\partial F_m}{\partial \tau},\frac{\partial F_m}{\partial \theta},\frac{\partial F_n}{\partial \theta})$ is if and only if \eqref{eq:alg-sign} is. Since $\epsilon_1=+1$ the signs agree. For $F_2=F_m$ we have $(e_s,e_\tau,\frac{\partial F_n}{\partial \theta},e_s+w,\frac{\partial F_m}{\partial \tau},\frac{\partial F_m}{\partial \theta})$, which is positive if and only if \eqref{eq:alg-sign} is negative (use three transpositions). Since in this case $\epsilon_1=-1$ the signs of $\Da$ and the Algorithm at the double point $x_i^j$ again agree.
\end{proof}

\begin{convention}\label{convention:Dax}
    The first sheet is depicted in \emph{red} and the second sheet in \emph{green}. If one more  double point is depicted, then its sheets are \emph{blue} and \emph{black} respectively. The moving sheet is denoted by a \emph{double} arrow tip, and the nonmoving sheet by a \emph{single} arrow tip. 
\end{convention}

In the next sections we will make extensive use of this algorithm for computing the Dax invariant, allowing us to express certain explicit knotted families in terms of grasper families. As a warm-up, let us apply the algorithm and this convention in the following proof.
\begin{lem}\label{lem:inverse}
    The class $\prealmap(-h)$ can be represented by the family that looks the same as for $\prealmap(h)$ except that either the arcs twirl from future to past, or there is a half-twist in the guiding finger as in Figure~\ref{fig:realmap-half-twist}, or the roles of the two meridians are exchanged as in Figure~\ref{fig:realmap-exchanged}.
\end{lem}
\begin{proof}
    The claim for the exchanged past and future is immediate from the Algorithm~\ref{alg}: the condition (3) is violated, resulting in the opposite sign.

\begin{figure}[!htbp]
    \centering
    \includegraphics[width=0.95\linewidth]{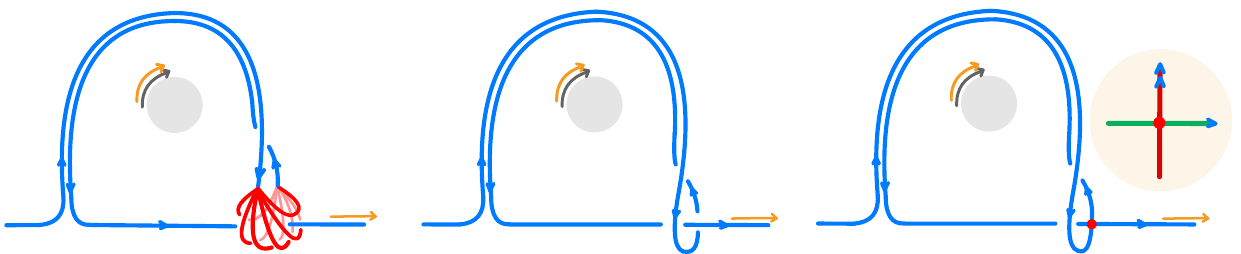}
    \caption{
        (i)~The family with a half-twist, representing $\prealmap(-h)$. 
        (ii)~The circle in this family completely contained in the present time, $\{0\}\times\R^3$. 
        (iii)~The only immersed circle in the trivialising family.
    }
    \label{fig:realmap-half-twist}
\end{figure}
    The half-twist in the finger corresponds to changing the orientations on all arcs in the foliation of the meridian sphere. Therefore, when computing $\Da$ of this family we get the same as for $\prealmap(h)$ (in Figure~\ref{fig:realmap-Dax}) except that the half-twist gives rise to the opposite orientation of the first (and moving) sheet, as in Figure~\ref{fig:realmap-half-twist}(iii). Therefore, $\epsilon_1=1$ and $\epsilon_2=-1$, so we obtain $-h$.

\begin{figure}[!htbp]
    \centering
    \includegraphics[width=0.9\linewidth]{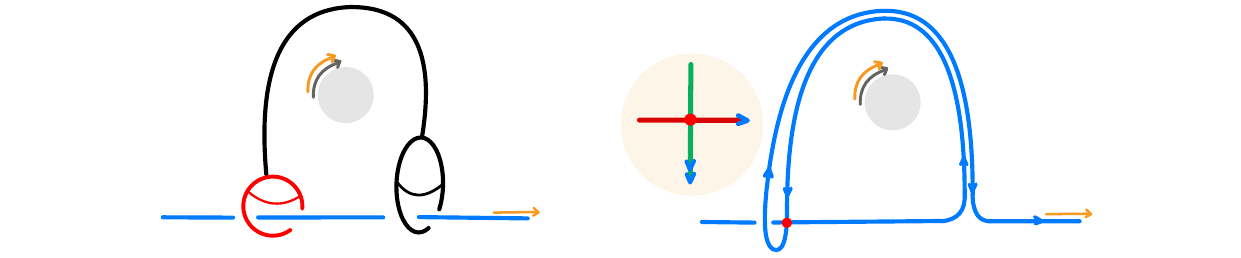}
    \caption{
        (i)~The family with exchanged meridians, representing $\prealmap(-h)$. 
        (ii)~The only immersed circle in the trivialising family.
    }
    \label{fig:realmap-exchanged}
\end{figure}
\begin{figure}[!htbp]
    \centering
    \includegraphics[width=0.9\linewidth]{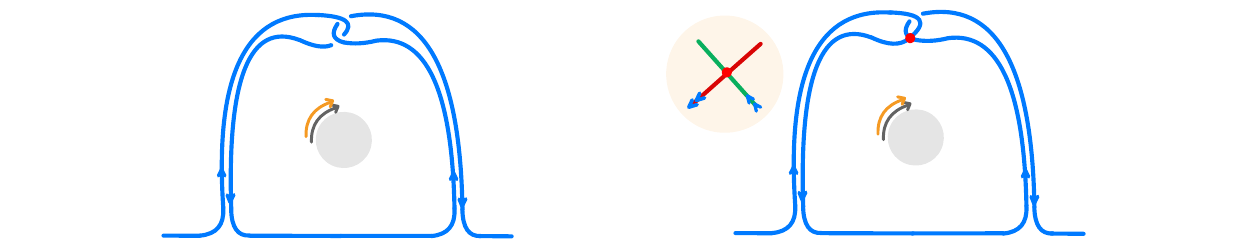}
    \caption{
        (i)~A circle in the symmetric family representing $\prealmap(h)$. The top of the finger on the left twirls from past to future, and the one on the right from future to past. 
        (ii)~The only immersed circle in the trivialising family. Note that both sheets are moving so the Algorithm has to be slightly adopted.
    }
    \label{fig:realmap-symmetric}
\end{figure}
    In Figure~\ref{fig:realmap-exchanged}(i) we have a picture as for $\prealmap(h)$ except that the roles of the two meridians are exchanged: now the piece of $c$ on the right is twirled around the sphere on the left. The Dax invariant is computed in Figure~\ref{fig:realmap-exchanged}(ii): the first sheet is now horizontal and is not moving, but the basis is still ``($\tau$-derivative, moving, nonmoving)'' as for $\prealmap(h)$. Thus  $\epsilon_1=-1$ and $\epsilon_2=1$, so we obtain $-h$.
\end{proof}

In particular, the first claim of the lemma implies that adding a full twist into the finger does not change the homotopy class of the family. The second claim gives rise to the following observation.
\begin{lem}\label{lem:symmetric}
   The class $\prealmap(h)$ can be represented by switching the roles of the meridians, and also changing the orientation of one meridian. Moreover, it can also be represented by the following more symmetric family: neighbourhoods of both $p_0$ and $p_1$ are moving, each on a sphere that has one hemisphere in the past, one in the future, and in the present we see the picture as in Figure~\ref{fig:realmap-symmetric}.
\end{lem}

\subsection{The dax homomorphism}\label{subsec:dax}
The kernel of $\prealmap$ can be expressed using the Dax homomorphism 
\[
\dax_u\colon\pi_3(X\sm\D^4)\to\Z[\pi_1X\sm1]
\]
which is an invariant of self-isotopies of arcs. Namely, we first remove a small ball around $c(e)\in X$ to obtain $u\colon\D^1\hra X\sm\D^4$. Then $\dax_u(a)$ counts double points of arcs that form a 2-parameter family that foliates the given $a\colon\S^3\to \X\sm\D^4$, based at $u$. We refer to \cite{K-Dax} for details.

    In particular \cite[Cor.1.4,Lem.4.9]{K-Dax} shows that for any $h\in\pi_1\X$ we have
    \begin{equation}\label{eq:the-formula}
        h^{-1}+h\bc^{-1}\in \ker(\prealmap).
    \end{equation}
    Since also $1\in\ker(\prealmap)$ by \cite{KT-highd}, we also have $\bc^k\in\ker(\prealmap)$ for all $k\in\Z$.

\begin{lem}\label{lem:g-in-kerr}
    If $\X=\M\#\S^1\tm\S^3$ and $c=\S^1\tm\{pt\}$, then we have
    \[
        \Z[\pi_1\M]<\ker(\prealmap).
    \]
\end{lem}
\begin{proof}
    In other words, if $g\in\pi_1\M$ we need to show that $\prealmap(g)=0$. Note that $G=\{pt\}\tm\S^3$ is an embedded 3-sphere that intersects $c$ in exactly one point. Thus, $G$ is the union of a small meridian ball of $c$ and a ball $B\subset\X\sm c$. Moreover, the leaf of $\prealmap(g)$ is precisely a meridian sphere $\partial B$, and the guiding arc of $\prealmap(g)$ follows $g\in\pi_1M$, which can be assumed to miss $G$. This means that we can undo the family $\prealmap(g)$ using $B$, since pushing our arcs through $B$ cannot create self-intersections.
\end{proof}

\section{Knotted families from general graspers}
\label{sec:general-graspers}

Note that for our simple grasper $\TG$ in a 4-manifold $X$ on $c\colon\S^1\hra \X$, when we defined the family $\prealmap(\pm h)$ we used two small meridian spheres to $c$, a root and a leaf. The self-referential families considered in~\cite{Gabai-disks,Budney-Gabai} generalise this by allowing the leaf to link $c$ more than once. More generally, we can allow the leaf sphere to be any embedded 2-sphere $L$ in $\X\sm\nu c$, and the root to be an embedded 3-ball $\wh{R}$ in $\X$ disjoint from $L$ and intersecting $c$ in finitely many points. To makes sense of the grasper family we require that the leaf is framed and parameterised, so that we know how to foliate it; that is, we fix a foliation of $\S^2$ and use the embedding $L$ to foliate $L(\S^2)$.

\begin{defn}\label{def:general-grasper}
   A  \emph{general grasper} $\TG$ in a 4-manifold $X$ on $c\colon\S^1\hra \X$ is the object that consists of an embedding $\wh{R}\colon\D^3\hra\X$, called the \emph{root} ball, that intersects $c$ transversely in finitely many points, an embedding $L\colon \S^2\hra \X\sm\nu c$, called the \emph{leaf} sphere, together with an embedded arc that connects $\wh{R}$ and $L$ and has interior disjoint from them, called the \emph{bar}. 
   
   Moreover, we assume that $\TG$ is in fact a thickening of all these objects, i.e.\ their trivialised tubular neighbourhood; in particular $\nu\wh{R}$ intersects $c$ in finitely many arcs, and $\nu L$ is a framed sphere.

   The associated grasper family $\prealmap^{\TG}\in\pi_1(\Emb(\S^1,\X);c)$ is defined similarly as in Definition~\ref{def:realmap}, by twirling simultaneously each arc $\nu\wh{R}\cap c$ around a parallel copy of $L$ in $\nu L$.
\end{defn}

Inspired by the clasper theory (compare with simple leaves in \cite{Habiro}), we introduce the following.

\begin{defn}\label{def:root-leaf}
    Let $\TG$ be a general grasper in $\X$ on $c$ with the root ball $\wh{R}$ and the leaf sphere $L$.
    \begin{itemize}
        \item  $\wh{R}$ is \emph{simple} if it intersects $c$ in exactly one point. 
        \item  $L$ is \emph{null} if it is the boundary of a map $\wh{L}\colon\D^3\to \X$.
        \item  $L$ is \emph{unknotted} if it bounds an embedded ball $\wh{L}\colon\D^3\hra \X$.
        \item  $L$ is \emph{semisimple} if it bounds an embedded ball $\wh{L}$ that does not intersect the bar of $\TG$.
        \item  $L$ is \emph{simple} if it bounds an embedded ball $\wh{L}$ that does not intersect the bar of $\TG$ and intersects $c$ in exactly one point. 
        \qedhere
    \end{itemize}
\end{defn}

Note that if a general grasper has both root and leaf simple, then it is isotopic to a simple grasper that gives an equivalent family of circles.
Moreover, since $\pi_1(\Emb(\S^1,\X);c)$ is an extension~\eqref{eq:Dax-extension} of $\im(\prealmap)$ by $\pi_1(\Imm(\S^1,\X);c)$, each general grasper family has a description in terms of simple graspers and also classes in $\pi_2\X$ and $\pi_1\X$. This extension is usually nontrivial~\cite{KT-4dLBT}, and we leave it to future work the study of these families and their relations.

In this section we consider two particular cases of general graspers: in Section~\ref{subsec:selfref} those with semisimple leaves, called \emph{semisimple graspers}, and in Section~\ref{subsec:simple-null-graspers} those with simple roots and null leaves, called \emph{simple-null graspers}. We express them as linear combinations of simple graspers, that is, in terms of $\prealmap$ (see Theorems~\ref{thm:selfref} and~\ref{thm:simple-null-graspers}). In the next sections these classes will be related to constructions of Watanabe, Budney--Gabai, and Gay--Hartman.
 

\subsection{Semisimple graspers}
\label{subsec:selfref}

The self-referential grasper of Gabai~\cite[Fig.9]{Gabai-disks} is the general grasper depicted in Figure~\ref{fig:sref}(ii): the root is at the tip of the blue finger and is simple, whereas the leaf sphere is a meridian of the finger and is semisimple. If the group element determined by the bar is $h\in\pi_1M$, we denote this family by $\sref^\circlearrowright(h)$ -- as only the ``top'', the leaf, is not simple.

In Figure~\ref{fig:sref}(i) the picture is mirrored, but we will see that this family $\sref^\circlearrowleft(h)$ is isotopic to $\sref^\circlearrowright(h)$.

Next, in Figure~\ref{fig:sref2}(i) we have a picture of $\sref_\circlearrowright(h)$. This is similar to Figure~\ref{fig:sref}(i), but here the root is nonsimple instead: the two parallel strands (a part of finger's body) twirl around the leaf sphere which is now simple; note that the tip of the finger stands still.

Finally, the family $\sref^\circlearrowright_\circlearrowright(h)$ from  Figure~\ref{fig:sref2}(ii) has both root and leaf nonsimple.

\begin{thm}\label{thm:selfref}
    The five families $\sref^\circlearrowleft(h)$, $\sref^\circlearrowright(h)$, $\sref^\circlearrowright_\circlearrowright(h)$, $\sref_\circlearrowright(h)$, $\prealmap(h+h^{-1})$ are all isotopic, i.e.\ define the same class in $\pi_1(\Emb(\S^1,\X);c)$.
\end{thm}
The proof is below. By \eqref{eq:the-formula} we have the relations $\prealmap(h^{-1})=\prealmap(-ht^{-1})$ and $\prealmap(1)=0$. In particular, $\prealmap(t^{-1})=\prealmap(1)=0$, and we derive the following result, that will be used later.
\begin{cor}\label{cor:selfref}
    We have $\sref^\circlearrowright_\circlearrowright(1)=\sref^\circlearrowright(1)=\prealmap(2)=0$  and $\sref^\circlearrowright_\circlearrowright(t)=\sref^\circlearrowright(t)=\sref^\circlearrowright(t^{-1})=\prealmap(t)$.
\end{cor}
\begin{proof}[Proof of Theorem~\ref{thm:selfref}]
    Since the Dax invariant is the inverse of $\prealmap$, it suffices to show that the value of $\Da$ for each family is $h+h^{-1}$. We use the Algorithm~\ref{alg} and Convention~\ref{convention:Dax} from Section~\ref{subsec:Dax-invt} to compute these values.
    
    For $\sref^\circlearrowright(h)$ we have by \cite[506-509]{Gabai-disks} or \cite[Rem.2.9]{K-Dax} (see also \cite[Fig.4]{K-Dax}):
\[
    \Da(\sref^\circlearrowright(h))= h+h^{-1}.
\]
    However, let us give an argument for completeness. A nullhomotopy of the family $\sref^\circlearrowright(h)$ is given by pulling apart the linking of the red sphere with the body of the finger, using the meridian ball. Two double points $x_1,x_2$ occur at a single time moment $j=1$, as in Figure~\ref{fig:sref-Dax}(ii). To compute the loops, note that the first sheet for $x_1$ is on the body of the finger, and $h_1=h^{-1}$, whereas the first sheet for $x_2$ is on the tip of the finger with the loop $h_2=h$. Both signs are positive, see the caption of Figure~\ref{fig:sref-Dax}(ii).
\begin{figure}[!htbp]
    \centering
    \includegraphics[width=0.95\linewidth]{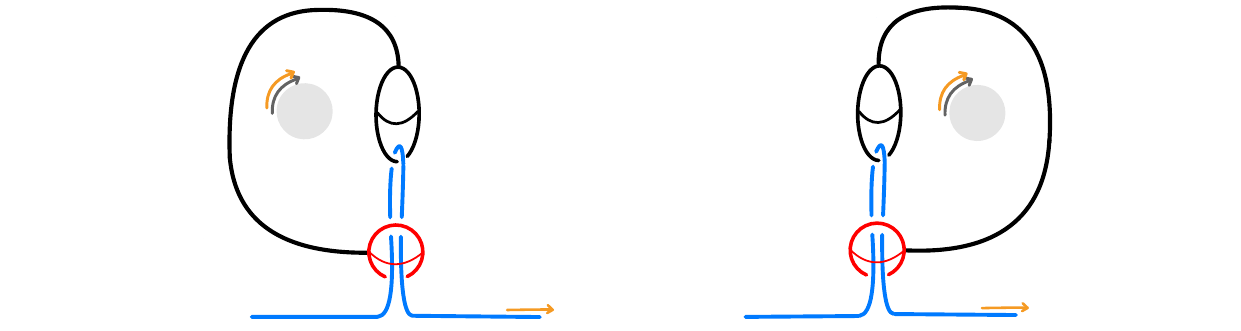}
    
    \vspace{20pt}
    
    \includegraphics[width=0.95\linewidth]{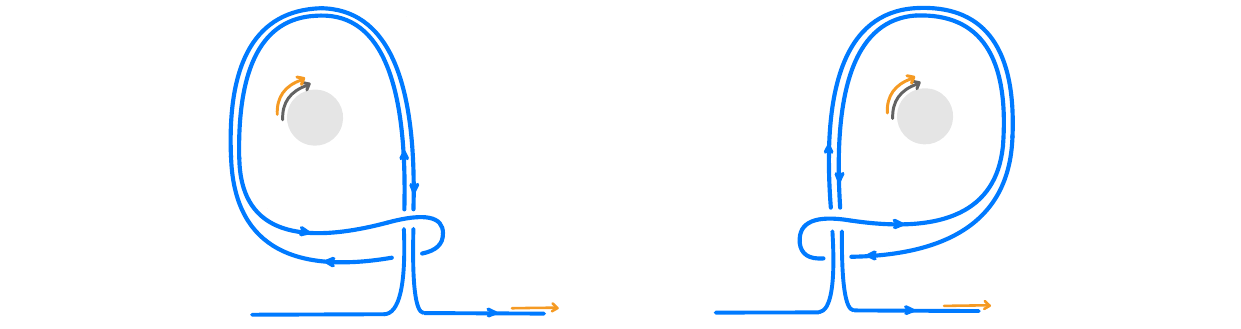}
    \caption{
        The semisimple families $\sref^\circlearrowleft(h)$ and $\sref^\circlearrowright(h)$ for some $h\in\pi_1(\X)$, and their key members.
    }
    \label{fig:sref}
\end{figure}
\begin{figure}[!htbp]
    \centering
    \includegraphics[width=0.95\linewidth]{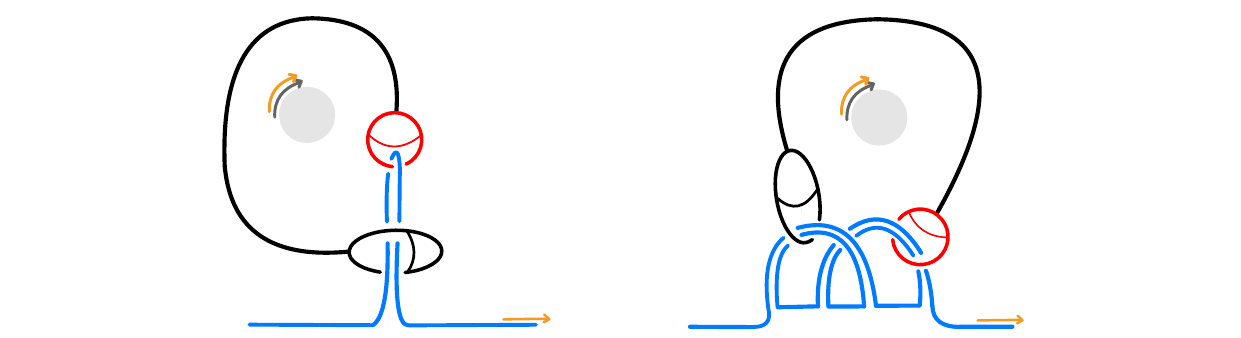}

    \vspace{20pt}
    
    \includegraphics[width=0.95\linewidth]{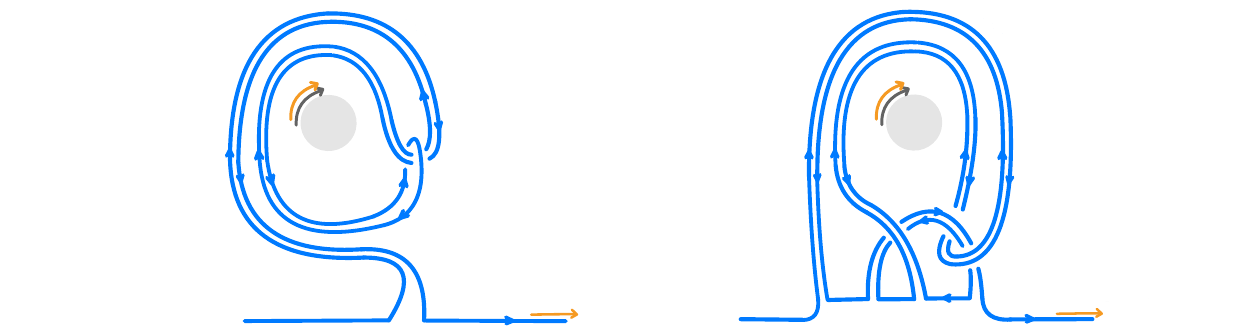}
    \caption{
        The semisimple families $\sref_\circlearrowright(h)$ and $\sref^\circlearrowright_\circlearrowright(h)$ for some $h\in\pi_1(\X)$, and their key members.
    }
    \label{fig:sref2}
\end{figure}

    Completely similarly we compute
    \[
        \Da(\sref^\circlearrowleft(h))= h+h^{-1}.
    \]
    Indeed, the local picture is precisely the same, see Figure~\ref{fig:sref-Dax}(i), so the signs agree with the previous case, and the loops are again $h_1=h$ and $h_2=h^{-1}$.
\begin{figure}[!htbp]
    \centering
    \includegraphics[width=0.95\linewidth]{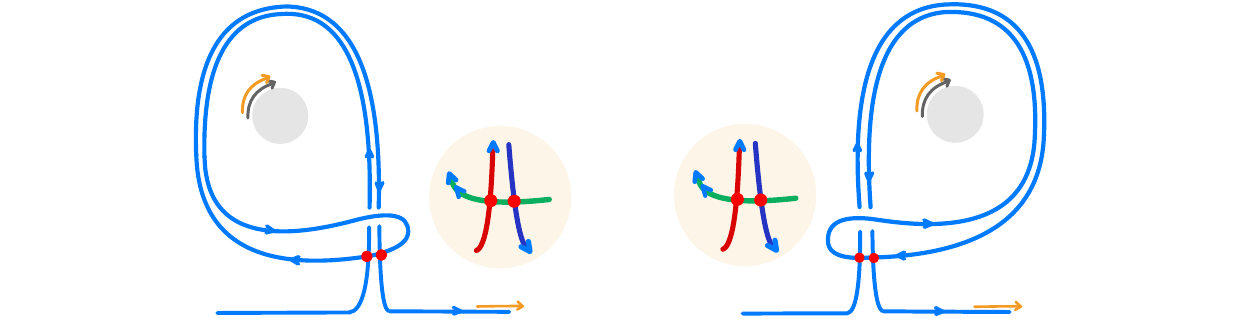}
    \caption{
        The Dax invariants of $\sref^\circlearrowleft(h)$ and $\sref^\circlearrowright(h)$ respectively. In each zoomed-in picture, the derivative of the moving sheet in the $\tau$-direction (the first vector in our basis of $\{0\}\times\R^3$) points towards the reader, so the desired basis will be positive if and only if one goes counterclockwise from the moving to the non-moving sheet; moreover, depicted arcs appear in $\S^1$ in the order (red, green, blue, black). 
        In (ii) at $x_1=$red$\cap$green the first sheet is not moving, and the basis is negative, so $\epsilon(x_1)=(-1)(-1)=1$, whereas $x_2=$green$\cap$blue has the first sheet moving and a positive basis, so $\epsilon(x_2)=(+1)(+1)=1$.
    }
    \label{fig:sref-Dax}
\end{figure}
\begin{figure}[!htbp]
    \centering
    \includegraphics[width=0.95\linewidth]{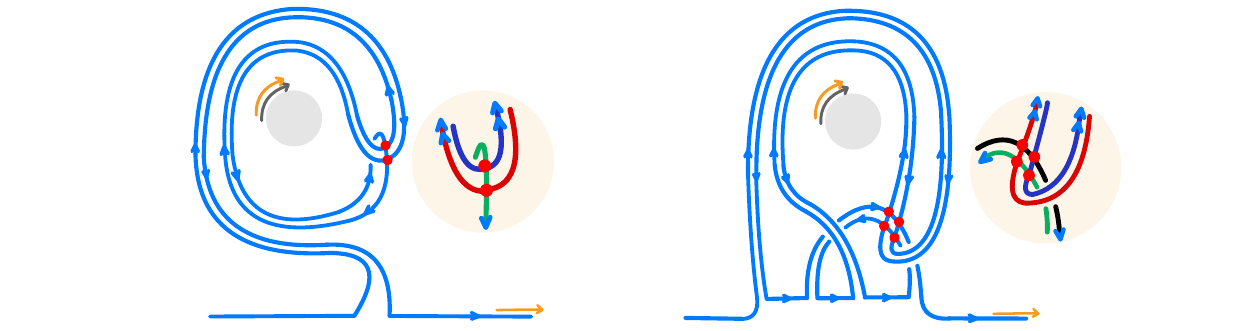}
    \caption{
        The Dax invariants of $\sref_\circlearrowright(h)$ and $\sref^\circlearrowright_\circlearrowright(h)$ respectively. In both cases at $x_1=$red$\cap$green the first sheet moves, and the basis is positive, so $\epsilon(x_1)=(+1)(+1)=1$. 
        For $x_2=$green$\cap$blue in (i) and $x_3=$green$\cap$blue in (ii), the first sheet is not moving, and the basis is negative, so $\epsilon=(-1)(-1)=1$. 
        Finally, in (ii) $x_2$ and $x_4$ are both on the black non-moving strand, but the directions of the other respective vector (red and blue) are mutually opposite, so we have $\epsilon(x_3)=1$ and $\epsilon(x_4)=-1$.
    }
    \label{fig:sref-Dax2}
\end{figure}

    Next, we use a similar nullhomotopy that pulls through the red meridian ball to see that
\[
    \Da(\sref_\circlearrowright(h)) = h+h^{-1}.
\]
    Indeed, in Figure~\ref{fig:sref-Dax2}(i) we see two double points, which when zoomed in and rotated look like in the previous case, except for which sheets move. But this changes the signs by $(-1)(-1)=1$, so they are both again positive. For the loops, at $x_1$ we have $hhh^{-1}=h$, and $hh^{-1}h^{-1}=h^{-1}$ at $x_2$.

    Finally, let us compute
\[
    \Da(\sref^\circlearrowright_\circlearrowright(h)) = h+h^{-1}.
\]
    Pulling through the red meridian ball now gives four double points, as in Figure~\ref{fig:sref-Dax2}(ii). For $x_1$ and $x_3$ (which are on the lower strand of the band on the right) the associated loops and signs look completely the same as in the previous case (Figure~\ref{fig:sref-Dax2}(i)), and give $h$ and $h^{-1}$. For both $x_2$ and $x_4$ the loop is $h$ followed by $hh^{-1}=1$ (which comes from traversing the first finger completely). We have $\epsilon(x_1)=-\epsilon(x_2)=-1$ since the picture is the same except that one vector is reversed. The same happens when comparing $x_3$ and $x_4$, but additionally the other sheet is the one moving, so $\epsilon(x_3)=(-1)(-1)\epsilon(x_4)=1$. Hence, $   \Da(\sref^\circlearrowright_\circlearrowright(h)) = h-h+h^{-1}+h=h+h^{-1}$.
\end{proof}

\begin{rem}[Splitting the leaf]\label{rem:splitting-leaf}
    One could also consider more general graspers $\TG$ in which the leaf sphere $L$ is unknotted in $\X$, not only in the complement of the bar. However, the resulting family of circles again reduces to a sum of simple graspers, by the following  \emph{splitting-leaf move}.
    
    Let $\wh{L}$ be a choice of a bounding 3-ball for $L$ in $X$. Then we can undo the family given by $\TG$ by pulling the arcs through $\wh{L}$. For each intersection point of the bar with $\wh{L}$ we get the Dax invariant $h+h^{-1}$ as in the last proof (so the family is $\prealmap(h+h^{-1})$) and for intersections of $c$ with $\wh{L}$ we get the Dax invariant $h'$ (so the family is $\prealmap(h')$), for some $h,h'\in\pi_1\X$. Thus, the family given by $\TG$ is the sum, in the Dax subgroup of $\pi_1(\Emb(\S^1,\X);x)$, of the families given by graspers $\TG_i$. These have leaves $L_i$ with bounding 3-balls $\wh{L}_i$, whose union gives the ball $\wh{L}$.  
\end{rem}

\begin{rem}\label{rem:isotopies}
    The reader might wonder if one can also see explicit isotopies between all mentioned families in Theorem~\ref{thm:selfref}; we give a sketch to satisfy the curiosity.

    Firstly, one can prove the \emph{splitting-leaf move} by an explicit isotopy: we first write $L$ as an ambient connect-sum of spheres that each links either the bar or $c$ exactly once; then the twirling motion around $L$ can be split into the sum of twirls around each of these spheres. Here we can choose to take the sum either in the $\pi_1$-direction, or in the $c$-direction; compare with Remark~\ref{rem:half-twist}.

    For example, we construct an isotopy from $\sref^\circlearrowright(h)$ to $\prealmap(h+h^{-1})$. In the former, our arc moves on the red sphere with time; equivalently up to an isotopy of the family, the arc can stop after one half of the sphere is traversed and go back to the basepoint, then from there return and traverse the other half of the red sphere. Thus, we are splitting the leaf sphere into the ambient connect-sum of two (oppositely oriented) meridian spheres to $c$. For the meridian on the right we get the same picture as for $\prealmap(h)$ whereas for the meridian on the left we have its mirror image, which is isotopic to $\prealmap(h^{-1})$, as one can check, see~\cite[Fig.3]{K-Dax}.

    An isotopy from $\sref^\circlearrowright(h)$ to $\sref^\circlearrowleft(h)$ simply rotates the tip of the finger while keeping the red sphere still; in other words in Figure~\ref{fig:sref} rotate the black bar around the $z$-axis. Alternatively, look at the key circles: rotating around the $z$-axis the picture at the bottom right of Figure~\ref{fig:sref} produces the picture on the bottom left, except that there is a half-twist in the finger, and the linking of the tip with the body is opposite. By Remark~\ref{rem:half-twist} these two facts cancel each other.
    
    To go from $\sref^\circlearrowleft(h)$ to $\sref_\circlearrowright(h)$, first observe an isotopy of the key circle at the bottom left of Figure~\ref{fig:sref2}, that pulls the tip of the moving finger back around $h$ until it lies vertical. The resulting picture is similar to the bottom left of Figure~\ref{fig:sref}, except that the linking of the tip with the body is opposite, and there is a difference in which sheets move and which stand still. These two signs cancel as in Lemma~\ref{lem:inverse} (the movement is exchanged at the cost of re-orienting the meridian).
    
    Finally, we can split the red sphere in $\sref^\circlearrowright_\circlearrowright(h)$ into the ambient connect-sum of the spheres that are meridian to the two arcs of the band on the right. Then observe that the upper one gives trivial family since the tip of the left finger is free to be isotoped back. On the other hand, the lower sphere gives the family similar to $\sref_\circlearrowright(h)$, except that there is a full twist in the finger. But this does not matter by Remark~\ref{rem:half-twist}, so $\sref^\circlearrowright_\circlearrowright(h)$ is isotopic to $\sref_\circlearrowright(h)$.
\end{rem}

\subsection{Simple-null graspers}
\label{subsec:simple-null-graspers}

Let us now consider a \emph{simple-null} grasper $\TG$ in $\X$ on $c$: this is a general grasper as in Definition~\ref{def:general-grasper} that has a simple root ball $\wh{R}$ (that is, $\wh{R}$ intersects $c$ in a single point), whereas the leaf sphere is a 2-knot $L\colon\S^2\hra \X\sm\nu c$ which is \emph{nullhomotopic} in $\X$. In particular, $L$ bounds a map $\wh{L}\colon\D^3\to \X$. 

Let us assume that the endpoint of the bar is the basepoint $L(e)\in L$, and that this is located in a fixed 3-dimensional ball which contains $c(e)$ and is drawn in our pictures. Thus, we can think of $L$ as a based sphere $[K]\in\pi_2\X$, using any whisker that goes from $L(e)$ to $c(e)$ in this ball.
\begin{defn}\label{def:bar-word-and-element}
    For a simple-null grasper $\TG$ define the \emph{bar word}
    \[
        \W\in\pi_1(\X\sm\nu L)
    \]
    as the homotopy class of the loop that traverses the bar and then the whisker. Moreover, define the \emph{bar group element} 
    $\bw=i_L(\W)\in\pi_1\X$
    as the image of the bar word under the map
    \[
    i_L\colon \pi_1(\X\sm\nu L)\to \pi_1\X
    \]
    induced by the inclusion.
    We denote this simple-null grasper by $\TG_{\W L}$.
\end{defn}

\begin{thm}\label{thm:simple-null-graspers}
    Let  $\TG_{\W L}$ be the simple-null grasper on $c\colon\S^1\hra \X$, with the leaf $L$ and the bar word $\W\in\pi_1(\X\sm \nu L)$. Fix a map $\wh{L}\colon\D^3\to \X$ with $\wh{L}|_{\partial\D^3}=L$. Then in $\X_{\nu L}=(\X\sm\nu L)\cup (\D^3\times\S^1)$ we have the union of $\wh{L}$ and the added $\D^3\times\{pt\}$ that gives a map $\wh{\wh{L}}\colon\S^3\to \X_{\nu L}$.

    The simple-null grasper family $\prealmap^{\TG_{\W L}}\in\pi_1(\Emb(\S^1,\X);c)$ is equal to
    \[
        \prealmap^{\TG_{\W L}}
        =\prealmap\circ(i_L)_*
        \Big(
            \W\dax(\wh{\wh{L}})\W^{-1}
        -\lambdabar(\W\wh{\wh{L}},\W)
        -\ol{\lambdabar(\W\wh{\wh{L}},\W)}
        \Big).
    \]
    Moreover, if $\wh{L}$ is an embedded ball, the first term vanishes.
\end{thm}

\begin{proof}
    First note that the family $\prealmap^{\TG_{\W L}}$ is obtained by foliating an embedded sphere (replace $\bw$ by a tube that guides $c$ into $L$). Moreover, since the grasper is away from the basepoint $c(e)$, we can keep a neighbourhood $\nu_e\subset c$ of $c(e)$ fixed throughout the family.
    This means that the projection of $\iota\prealmap^{\TG_{\W L}}\in\pi_1(\Imm(\S^1,\X);c)$ further to $\Fix_{\bc}(\pi_1\X)\leq\pi_1\X$ is trivial, see \eqref{eq:Dax-extension}. Thus, the class in $\pi_1(\Imm(\S^1,\X);c)$ comes from $\pi_2\X/b=\bc\cdot b$, and it is easy to see that the corresponding sphere is $[\bw K]\in\pi_2\X$. But by our assumption this is trivial.

    Therefore, \eqref{eq:Dax-extension} tells us that $\prealmap^{\TG_{\W L}}$ is in the image of $\prealmap\colon\Z[\pi_1\X]/\ker(\prealmap)\to\pi_1(\Emb(\S^1,\X);c)$. To find the class in $\Z[\pi_1\X]$ from which it comes, we compute the Dax invariant from Section~\ref{subsec:Dax-invt}:
    \[
    \Da(\prealmap^{\TG_{\W L}})\in \faktor{\Z[\pi_1\X]}{\ker(\prealmap)}.
    \]
    For this we need a nullhomotopy $F$ of our family, and this can be produced by foliating the based 3-ball $\bw \wh{L}$ by arcs. Note that the choice of $\wh{L}$ does not matter since any other choice $\wh{L}'$ glues with $\wh{L}$ to a 3-sphere in $\X\sm\D^4$, so $\dax_u(\bw \wh{L}\cup\bw \wh{L}')\in \ker(\prealmap)$ for the Dax homomorphism from Section~\ref{subsec:dax}.
    
    Moreover, $c$ survives to the surgered manifold $\X_{\nu L}$, so we still have $u\colon\D^1\hra \X_{\nu L}\sm\D^4$, and we can compute $\dax_u(\W\wh{\wh{L}})\in\Z[\pi_1\X_{\nu L}\sm1]$. Since $\D^3\times\{pt\}$ is embedded, double points that occur in $\dax_u(\bw\wh{\wh{L}})$ are precisely those that occur when computing $\Da(\prealmap^{\TG_{\W L}})$ as in the previous paragraph. In addition, since  $\pi_1(\X\sm\nu L)\cong\pi_1(\X_{\nu L})$ (as the difference is $\S^1\times\D^3$) the map $i_L\colon \X\sm\nu L\to \X$ induces a ring homomorphism $(i_L)_*\colon\Z[\pi_1\X_{\nu L}\sm1]\to\Z[\pi_1\X\sm1]$. Thus, the preceding observation implies 
    \[
        \Da(\prealmap^{\TG_{\W L}})=(i_L)_*(\dax_u(\W\wh{\wh{L}})).
    \]
    Moreover, since $c=\barbemb c_1$ is a small meridian circle of $\barbemb S_1$, the arc $u$ is homotopic into boundary in $\X\sm\D^4$. In this case we denote $\dax\coloneqq\dax_u$ as in our previous work~\cite[Thm(II),(22)]{K-Dax}, where we determine the expression for $\dax(\W\wh{\wh{L}})$: for $a\in\pi_3(\X\sm\D^4)$ and $h\in\pi_1\X$ we show that 
    \[
        \dax(ha)=h\dax(a)h^{-1}
        -\lambdabar(ha,h)-\ol{\lambdabar(ha,h)}.
    \]
    Finally, if $\wh{L}\subset \X$ is embedded, so is $\wh{\wh{L}}\subset \X_{\nu L}$, implying that $\dax(\wh{\wh{L}})$ vanishes~\cite[Cor.B(i)]{K-Dax}.
\end{proof}

\begin{rem}
    The surgery on $\nu L$ in this theorem might seem unnecessary. After all, our goal is to compute the invariant $\Da$ in $\X$. However, once we pass to $\X_{\nu L}$ we can easily express this invariant in terms of $\lambda$, whereas in $\X$ we would have to deal with a relative version of $\lambda$ (for 3-manifolds with boundary). See Corollary~\ref{cor:implant-formula} for how this theorem is applied.
\end{rem}


\subsection{Borromean link families}
\label{subsec:links}

Let us briefly discuss some families of links in $\S^4$ that will be useful in Section~\ref{sec:W}. Firstly, we consider the \emph{Hopf link}
    \begin{equation}\label{eq-def:Hopf}
        \Hopf\colon\S^2\sqcup\S^1\hra \S^4
    \end{equation}
where $\Hopf(\S^1)$ can be viewed as a small meridian circle of $\Hopf(\S^2)$, and up to isotopy vice versa, $\Hopf(\S^2)$ can be viewed as a small meridian sphere of $\Hopf(\S^1)$. 

    \begin{figure}[!htbp]
        \centering
        \includegraphics[width=\linewidth]{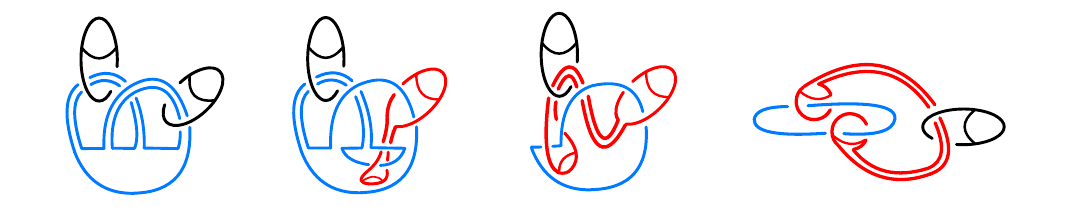}
        \caption{
            (i)~The Borromean link $\Bor$. 
            (ii-iv)~An isotopy of the Borromean link $\Bor$.
        }
        \label{fig:Bor}
    \end{figure}
Secondly, we consider the \emph{Borromean link}
    \begin{equation}\label{eq-def:Bor-link}
        \Bor\colon\S^1\sqcup\S^2\sqcup\S^2\hra \S^4
    \end{equation}
defined as follows, see Figure~\ref{fig:Bor}(i). We let $\Bor(\S^1)$ be the boundary of a standard genus one surface in $\S^4$ (obtained by taking a disk $\D^2$ and attaching two bands, i.e.\ 2-dimensional 1-handles), and we let the sphere components $\Bor(\S^2\sqcup\S^2)$ be the meridians of those bands. Then the isotopy depicted in Figure~\ref{fig:Bor} shows the following.
\begin{lem}
    The link $\Bor$ is isotopic to the one that consists of an unlink $\Unlink\colon\S^1\sqcup\S^2\hra \S^4$ together with a sphere $\S^2\hra \S^4\sm \Unlink$ obtained by ambiently connect-summing two meridians of $\Unlink(\S^1)$ along an arc that is a meridian to $\Unlink(\S^2)$.
\end{lem}
Let us now foliate one of the spheres in $\Bor$ (for example, the one on the right of Figure~\ref{fig:Bor}(i)) by a 1-parameter family of embedded circles, to define the \emph{Borromean linked family}
    \begin{equation}\label{eq-def:Bor-fam}
        \Bor_s\colon\S^1\sqcup\S^1\sqcup\S^2\hra \S^4,
    \end{equation}
where $s\in[0,1]$ and $\Bor_0=\Bor_1$.
To define the foliation, we take for the basepoint circle $\Bor_0(\S^1)$ the union a small arc $\alpha\subseteq\Bor(\S^2)$ around the north pole and an arc in $\S^4$ that intersects $\Bor(\S^2)$ in $\partial\alpha$ but is otherwise disjoint from  $\Bor(\S^2)$; we then twirl $\alpha$ around the sphere $\Bor(\S^2)$ while keeping the endpoints fixed. The last lemma implies:

\begin{cor}
    The Borromean linked family $\Bor_s\colon\S^1\sqcup\S^1\sqcup\S^2\hra \S^4$ is isotopic to the family obtained by foliating the red sphere in Figure~\ref{fig:Bor}(iv).
\end{cor}
Finally, we fix framings on the links $\Hopf$ and $\Bor$ and $\Bor_s$. In each case, both sphere components have trivial normal bundles, and since $\pi_2SO_2\cong0$ there is a unique isotopy class of framings. Each circle component has a trivial normal bundle, and two possible framings, since $\pi_1SO_3\cong\Z/2$; we pick the 0-framing (the one that extends across a bounding disk $\D^2\subset\S^4$). Note that this makes sense across the whole family $\Bor_s$: the family is an isotopy of the circle that extends to an isotopy of a bounding disk, so it preserves the 0-framing.

\section{Framed families and parameterised surgery}\label{sec:ps}
\label{subsec:framed}  

\subsection{Framed families}

Let $\X$ be any oriented 4-manifold with a fixed $\nu c\colon\nu\S^1\hra \X$. We lift all our simple grasper families along the map
\[
    \pi_1(\Emb(\nu\S^1,\X);\nu c) \ra \pi_1(\Emb(\S^1,\X);c) 
\]
that restricts along $\S^1=\S^1\times\{0\}\hra\S^1\times\D^3= \nu\S^1$, i.e.\ forgets the framing. 

Now, we rely on the study of framed embeddings from \cite{K-Dax2}.
In that paper we show that a class $[f]\in\pi_1(\Emb(\S^1,\X); c)$, such that $f_t$ are based embeddings (i.e.\ map the basepoint $e\in\S^1$ to the basepoint $c(e)\in \X$), \emph{does not lift} to a loop of framed embeddings if and only if $p([f])\in\pi_2\X$ has nontrivial second Stiefel--Whitney number. Since all classes in the image of $\prealmap$ fix the basepoint and have vanishing $p$, we conclude that they do lift. 

Moreover, for any manifold $\X$ and a given $\nu c$ there are at most two such lifts. 
In one of them a neighbourhood of the basepoint of $c$ is fixed throughout the family, whereas in the other all normal 3-balls rotate, generating $\pi_1SO_3=\Z/2$.
Our preferred lift is the stationary framing.

\begin{cor}
    The map $\prealmap$ has a lift
    \[
        \Z/2\times \faktor{\Z[\pi_1\X]}{\ker(\prealmap)}\ra \pi_1(\Emb(\nu\S^1,\X);\nu c).
    \]
    We keep using the notation $\prealmap$ for the restriction of this map to $\{0\}\times \Z[\pi_1\X]/\ker(\prealmap)$.
\end{cor}

Similarly, each of the families $\sref$ and $\prealmap^{\TG}$ admits a canonical lift to a family of framed embeddings.

\begin{rem}\label{rem:emb-torus-gen}
    In Remark~\ref{rem:emb-torus} we saw that simple grasper families are obtained by foliating embedded tori. All simple-null grasper families $f$ also have embedded representatives: define $T_f$ as the torus obtained from the thin torus containing $c$ by a connect-sum with $L$ using the bar as the guiding arc.
    Moreover, note that $T_f$ is framed, and a framing of $f$ can be obtained by foliating $\nu T_f$, and adding the normal framing of $f_a$ in $T_f$.
\end{rem}

\subsection{The maps}
Let us now recall~\eqref{eq-intro:ps-S} from the introduction, the parameterised surgery map
\begin{equation}\label{eq-def:ps-S}
\begin{tikzcd}[column sep=1.8cm]
    \ps_{\nu S}\colon \pi_1(\Emb(\nu\S^1,\M_{\nu S});\nu c) \rar{\delta_{\nu c}} & \pi_0\Diffp(\M\sm\nu S) \rar{-\cup\Id_{\nu S}} & \pi_0\Diffp(\M),
\end{tikzcd}
\end{equation}
This will be used in Section~\ref{sec:BG}. The case when $S$ is unknotted will be used in Section~\ref{sec:W}, where $\ps_{\nu S}$ is denoted as in~\eqref{eq-intro:ps}, simply by
\begin{equation}\label{eq-def:ps}
\begin{tikzcd}[column sep=1.5cm]
    \ps\colon \pi_1(\Emb(\nu\S^1,\M\#\,\S^3\times\S^1);\nu c) \rar{\delta_{\nu c}} & \pi_0\Diffp(\M\sm\nu S) \rar{-\cup\Id_{\nu S}} & \pi_0\Diffp(\M).
\end{tikzcd}
\end{equation} 
Recall that $\delta_{\nu c}$ lifts a loop of framed $f_t\colon\S^1\hra \M_{\nu S}$ based at $f_0=f_1=\nu c$ to a path $F_t$ of diffeomorphisms of the surgered manifold $\M_{\nu S}\coloneqq (\M\sm\nu S)\cup \nu c$ with $F_0=\Id$ and $F\circ f_t=F_t$ (i.e.\ $F$ is an ambient isotopy extending $f$), then restrict the endpoint diffeomorphism $F_1$ to the complement of $\nu c$, so $\delta_{\nu c}(f)=F_1|_{\M_{\nu S}\sm\nu c}$. We then extend this by the identity over $\nu S$, in order to obtain a diffeomorphism of $\M=(\M_{\nu S}\sm\nu c)\cup\nu S$.

\begin{rem}\label{rem:amb}
    Consider the map $\ev_{\nu c}\colon\Diffp(\M_{\nu S})\to \Emb(\nu\S^1,\M_{\nu S})$ that evaluates a diffeomorphism $\varphi$ of $\M_{\nu S}$ on $\nu c\subseteq \M_{\nu S}$ (i.e.\ sends $\varphi$ to $\varphi\circ\nu c$). This is a fibre bundle by a classical result of Cerf and Palais. The fibre consists of those $\varphi$ which fix $\nu c$, so are precisely diffeomorphisms of $\M_{\nu S}\sm \nu c=\M\sm\nu S$. Then the map $\delta_{\nu c}$ in~\eqref{eq-def:ps-S} is exactly the connecting map of this fibre bundle.
\end{rem}

The following is inspired by an idea of Peter Teichner; see also \cite[Sec.\ 4]{Gay} and \cite[Sec.\ 9]{Watanabe-theta}.

\begin{prop}\label{prop:phc}
    Diffeomorphisms of a 4-manifold $\M$ that are in the image of $\ps$ are all pseudo-isotopic to the identity. More precisely, this map factors as
    \[\begin{tikzcd}
        \ps\colon\;
        \pi_1(\Emb(\nu\S^1,\M\#\,\S^3\times\S^1);\nu c)\rar{\phc} & \pi_0\Diff_{\partial_0}(\M\times[0,1])\rar{\partial_1} & \pi_0\Diffp(\M)
    \end{tikzcd}
    \]
    where $\Diff_{\partial_0}(\M\times[0,1])$ is the group of pseudo-isotopies of $\M$, and $\partial_1$ is the restriction to the top boundary of a pseudo-isotopy. 
\end{prop}
\begin{proof}
    First recall that a pseudo-isotopy is a diffeomorphism of $\M\times[0,1]$ that is the identity on $\M\times\{0\}$ and $\partial \M\times[0,1]$. Let us define the map $\phc$, the ``parameterised handle cancellation''.

    We denote $\X\coloneqq \M_{\nu R}$ for a fixed small unknotted sphere $R$, and forget the identification of $\M_{\nu R}$ with $\M\#\,\S^3\times\S^1$ for the moment.
    Suppose we are given $f_t\colon \nu\S^1\hra \X$, $f_0=f_1=\nu c$, and let $F\colon[0,1]\to\Diffp(\X)$, $F_0=\Id$, $F_s\circ f_0=f_s$, be an ambient isotopy extension. We consider its track 
    \[
        \wt{F}\colon \X \times[0,1]\to \X\times[0,1], \quad (x,t)\mapsto (F_t(x),t).
    \]
    Note that this is a diffeomorphism that is the identity $\wt{F}_0=\Id$ on $\X\times\{0\}$, and on the top on $\X\times\{1\}$ we have $\wt{F}_1=F_1$.
    Let us consider the manifold 
    \[
        \M'\coloneqq (\M\times[0,\e]\sm \nu \wh{R}')\cup (\X\times[0,1]) \cup_{\nu c} H^2.
    \]
    Here $\wh{R}'$ is a 3-ball in $\M\times[0,\e]$ with boundary $\partial \wh{R}'=R\subset \M\times\{\e\}$, obtained by pushing into the interior a ball $\wh{R}\subset \M\times\{e\}$ bounded by our unknotted sphere $R$. In particular, $\M\times[0,\e]\sm \nu \wh{R}'$ is a cobordism from $\M$ to $\X=\M_{\nu R}$. To the top of this we can glue $\X\times[0,1]$. To the new top we then attach the 5-dimensional 2-handle $H^2=\D^2\times\D^3$ along $\nu c\subset \X\times\{1\}$. At the top we see the surgery $\X_{\nu c}=(\X\sm\nu c)\cup \D^2\times\S^2\cong(\M\sm\nu R)\cup\nu R= \M$.
    
    In fact, the whole cobordism is diffeomorphic to $\M\times[0,1]$ by the cancellation of $h^2$ and $h^1$, where the latter we had created by removing $\nu \wh{R}'$. Namely, we are using the 5-dimensional analogue of the ``dotted notation'' procedure in Kirby diagrams of 4-manifolds: we have
    \[
        \M\times[0,\e]\sm \nu \wh{R}'
        \cong (\M\times[0,\e/2]\cup h^1\cup h^2)\sm\nu \wh{R}'
        \cong \M\times[0,\e/2]\cup h^1,
    \]
    where $h^1\cup h^2$ are in cancelling position, and the last equality holds because $\nu \wh{R}'=\D^3\times\D^2$ is precisely $h^2=\D^2\times\D^3$ turned upside down. Finally, the 2-handle $H^2$ in $\M'$ is attached along $\nu c$, so exactly cancels this $h^1$. Let $\psi\colon \M\times[0,1]\to \M'$ be a fixed diffeomorphism, and note that we can assume $\Psi|_{\M\times\{0\}}=\Id$ and $\psi\coloneqq\Psi|_{\M\times\{1\}}\colon \X_{\nu c}\to \M$ is our chosen diffeomorphism that is the identity on $\X\sm\nu c=\M\sm\nu R$.

    Using the decomposition of $\M'$ we can extend the track $\wt{F}$ to a diffeomorphism 
    \[
        F'\coloneqq\;\Id_{\M\times[0,\e]\sm \nu \wh{L}'}
        \cup \wt{F}
        \cup_{\nu c}\Id_{H^2}\colon 
        \M'\times[0,1]\ra \M'\times[0,1].
    \]
    Note that these maps indeed glue together, because $\wt{F}_0=\Id$ and $\wt{F}_1$ preserves $\nu c$.
    Then we define the desired pseudo-isotopy 
    \[
        \phc(f)\coloneqq
        \Psi^{-1}\circ F'\circ\Psi\colon\; 
        \M\times[0,1]\ra \M\times[0,1].
    \]
    This is indeed $\Id$ on $\M\times\{0\}$, whereas on the top we see:
    \[
        \partial_1\phc(f)=\Psi^{-1}\circ F'_1\circ\Psi|_{\M\times\{1\}}
        =\psi^{-1}\circ(F_1|_{\X\sm\nu c}\cup\Id_{\D^2\times\S^2})\circ\psi,
    \]
    which agrees with
    $\ps(f)$, since $\psi$ is the identity on $\M\sm\nu R$ and $\psi^{-1}\circ\Id_{\D^2\times\S^2}\circ\psi=\Id_{\nu R}$.
%
%
\end{proof}

\section{Watanabe's theta classes}\label{sec:W}

In Section~\ref{subsec:theta-def} we describe Watanabe's classes, and in Section~\ref{subsec:W-ps} we relate them to grasper classes.

\subsection{Definition of a theta class}
\label{subsec:theta-def}

Let $\Theta$ denote the graph with two vertices $v,w$ that are connected by three edges $e_1,e_2,e_3$. An embedding $\Theta\hra \M\sm\partial \M$ is determined by two elements $g_1,g_2\in\pi_1\M$ up to isotopy, since $d=4$. See Figure~\ref{fig:graph}(i). We can isotope any embedding so that the middle edge $e_2$ is short (contained in a small ball in $\M$), as in Figure~\ref{fig:graph}(ii). We represent $g_i\in\pi_1\M$ as the loops that go around the shaded disks, and the basepoint $v$ of $\M$ by the full square.  
\begin{figure}[!htbp]
    \centering
    \includegraphics[width=\linewidth]{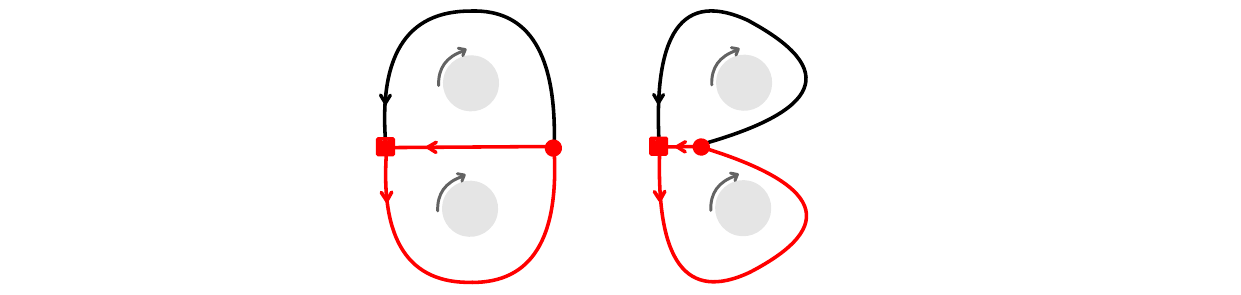}
    \caption{
        (i)~An embedded $\Theta$-graph, with associated group elements $g_1=e_1^{-1}\cdot e_2$ and $g_2=e_2^{-1}\cdot e_3^{-1}$. 
        (ii)~An isotopic embedding with $e_2$ short.}
    \label{fig:graph}
\end{figure}

Given such an embedding $\Theta_{g_1,g_2}$, Watanabe~\cite{Watanabe-dim-4} constructs a class
\[
    \wat(\Theta_{g_1,g_2})\in\pi_0\Diffp(\M).
\]
The following equivalent formulation of this construction was described by Botvinnik and Watanabe in~\cite{Botvinnik-Watanabe}. We first orient each edge, so that each vertex has at least one incoming and one outgoing edge (such a choice is always possible), as in Figure~\ref{fig:graph}(i).

\begin{figure}[!htbp]
    \centering
    \includegraphics[width=0.95\linewidth]{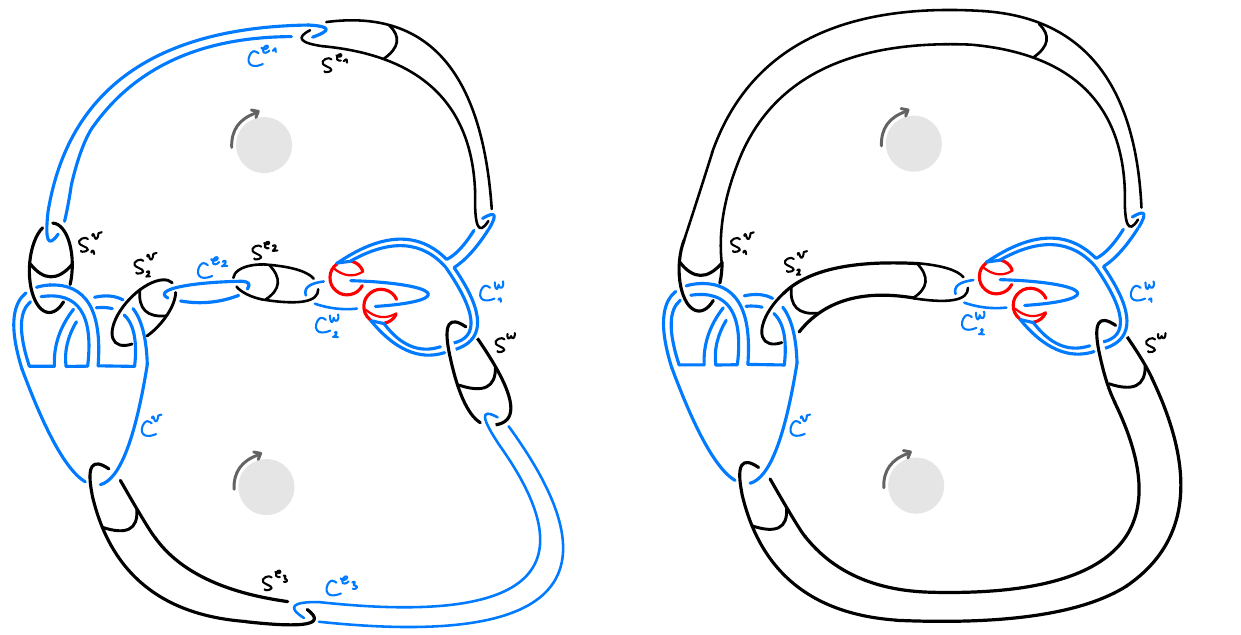}
    \caption{
        (i)~The link corresponding to the embedded $\Theta$-graph. 
        (ii)~The resulting link after the cancellation of the three Hopf link pairs corresponding to the edges.
    }
    \label{fig:theta1}
\end{figure}
Then, as in Figure~\ref{fig:theta1}(i) we replace each edge by the Hopf link $\Hopf$ from \eqref{eq-def:Hopf}, the vertex $v$ by the Borromean link $\Bor$ from \eqref{eq-def:Bor-link}, and the vertex $w$ by the Borromean family $\Bor_s$, $s\in\S^1$, from \eqref{eq-def:Bor-fam}, with  framings as in the end of Section~\ref{subsec:links}. Moreover, each half-edge indicates that a component of the Hopf link and one component of the vertex form a Hopf link as well. We obtain a family of framed $12$-component links
\[
    \Theta_{g_1,g_2}(s)\subset \M,
\]
parameterised by $s\in\S^1$ (where only one component depends on $s$).

Now consider the bundle whose fibre over $s\in\S^1$ is the surgery on the framed link $\Theta_{g_1,g_2}(s)\subset \M$. Similarly as in Proposition~\ref{prop:phc} this is the top of a bundle over $\S^1$ of 5-dimensional cobordisms obtained from $\M\times[0,1]$ by attaching 1- and 2-handles according to the given link; the belt spheres of 1-handles are the given 2-spheres, whereas the attaching circles of 2-handles are the given circles. For these handles we can use handle slides, cancellations and isotopies. The resulting bundle of surgeries will remain the same (up to bundle isomorphism) if these moves are done consistently throughout the family. 

For example, we can cancel the three Hopf links of spheres $S^{e_i}$ and circles $C^{e_i}$, one for each edge $i=1,2,3$, after we first slide the other strands linking the sphere over $C^{e_i}$. After an isotopy ``pulling to the left'' we obtain Figure~\ref{fig:theta1}(ii). See \cite{Botvinnik-Watanabe} for a detailed discussion.

Since $\Bor_s$ is trivial for a fixed $s$, one can isotope $\Theta_{g_1,g_2}(s)$ into a split collection of 0-framed Hopf links, so the result of surgery is diffeomorphic to $\M$. The classifying map for this $\M$-bundle over $\S^1$ determines  \emph{Watanabe's theta class}
\begin{equation}\label{eq-def:W}
    \wat(\Theta_{g_1,g_2})\in\pi_1B\Diffp(\M)\cong\pi_0\Diffp(\M).
\end{equation}

\subsection{Relation to parameterised surgery and graspers}
\label{subsec:W-ps}
We now express \eqref{eq-def:W} in terms of semisimple grasper classes using the parameterised surgery map~\eqref{eq-def:ps}.

\begin{thm}\label{thm:Theta}
    For a 4-manifold $\M$ and $\Theta_{g_1,g_2}\colon\Theta\hra \M$ there are equalities
    \[
        \wat(\Theta_{g_1,g_2})=\ps\big(
        \sref^\circlearrowright_\circlearrowright(g_1g_2^{-1}tg_2\big)
        -
        \sref^\circlearrowright_\circlearrowright(g_1) ).
    \]
\end{thm}
\begin{proof}
    We continue to cancel surgeries. After two handle slides of $C^v$ (the two strands in its right-hand side band) over $C_2^w$, we can cancel the Hopf pair $(S_2^v, C_2^w)$, and again isotope the result by pulling to the left. We obtain 4-component link shown in Figure~\ref{fig:theta2}(i).
    
    Similarly, now slide the other two strands of $C^v$ over $C_1^w$ in order to cancel the Hopf pair $(S_1^v,C_1^w)$, giving the 2-component link $C^v,S_2^v$ in Figure~\ref{fig:theta2}(ii); this is precisely the family \cite[Prop.\ 4.2]{Botvinnik-Watanabe}.
    \begin{figure}[!htbp]
        \centering
        \includegraphics[width=0.95\linewidth]{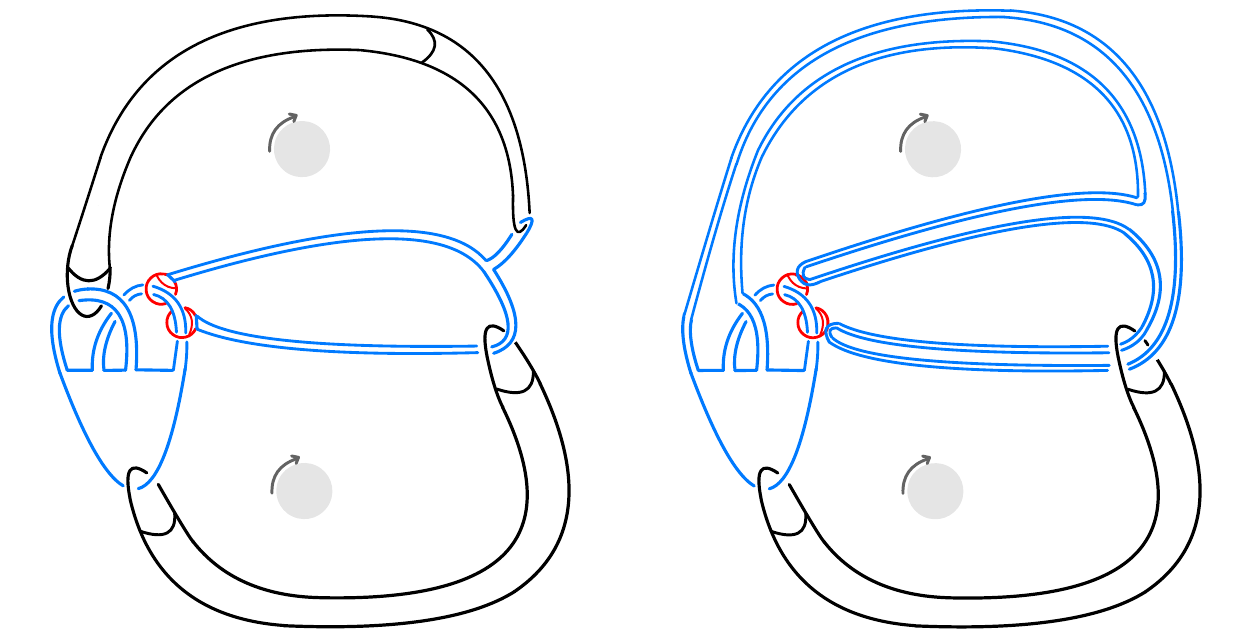}
        \caption{
            ~(i) After the cancellation of the Hopf link in the middle.
            ~(ii) After the cancellation of the Hopf link at the top.
        }
        \label{fig:theta2}
    \end{figure}

   Now, instead of forming the bundle of surgeries on the framed link $C^v(s)\sqcup S_2^v$, we can first do surgery only on $S_2^v$, and view $C^v(s)\colon\nu\S^1\hra\M\#\,\S^3\times\S^1$. The bundle of surgeries on $C^v(s)$ has the monodromy given as follows: lift $C^v$ to an ambient isotopy of $\M\#\,\S^3\times\S^1$, then remove $\nu c$ and add $\nu S$. Therefore, by definition of $\ps$ in \eqref{eq-def:ps} we have
    \[
        \wat(\Theta_{g_1,g_2}) = \ps(C^v).
    \]
    Moreover, we can now use the splitting-leaf move from Remark~\ref{rem:splitting-leaf} to write the family $C^v$ as the sum of families $f_1,f_2\colon\S^1\to\Emb(\nu\S^1,\M\sm \nu S_3)$ depicted in Figure~\ref{fig:theta3}(i) and (ii), using the general grasper notation of Section~\ref{subsec:graspers}. 
        \begin{figure}[!htbp]
        \centering
        \includegraphics[width=0.95\linewidth]{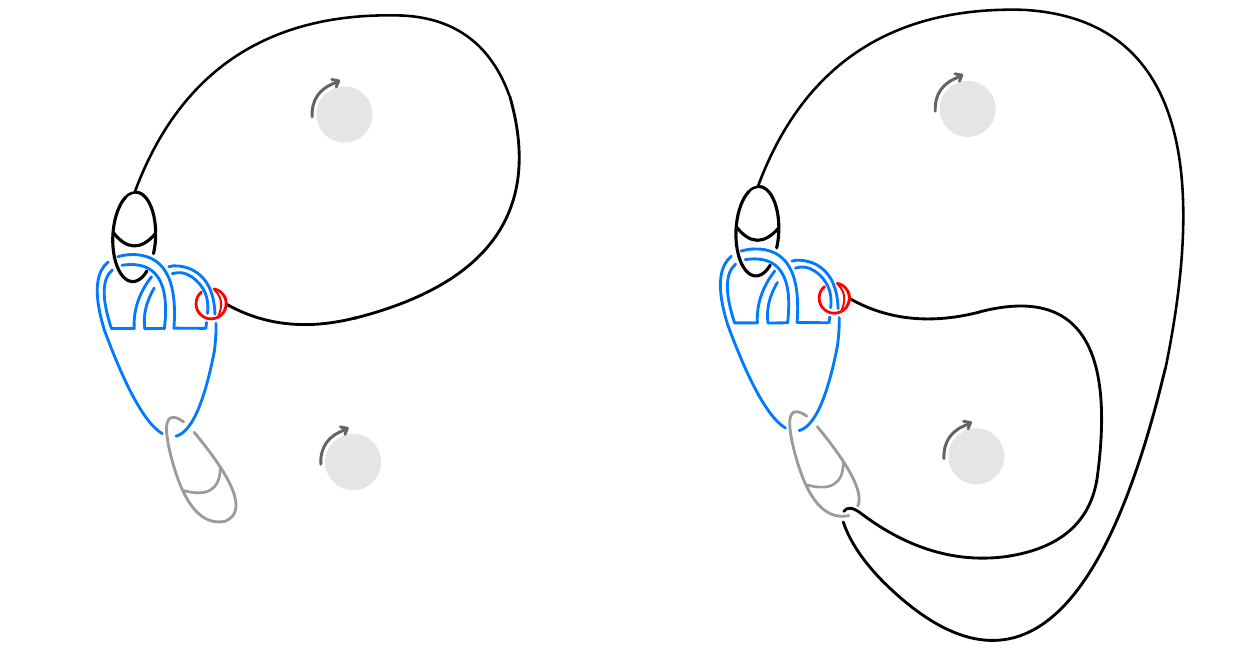}
        \caption{
            Watanabe's theta class is the parameterised surgery on the sum of two depicted knotted families. 
            (i)~The red sphere links the blue band oppositely to Figure~\ref{fig:sref2}(ii).
            (ii)~The red sphere links the blue band correctly, but now the bar links the sphere at the bottom.
        }
        \label{fig:theta3}
    \end{figure}
    
    Clearly, each $f_i$ is a family of the shape $\pm\sref^\circlearrowright_\circlearrowright(h)$, as in Figure~\ref{fig:sref2}, for some $h\in\pi_1(\M\#\,\S^3\times\S^1)$(ii). We next determine these group elements and signs.
    
    Firstly, for the family $f_2$ in Figure~\ref{fig:theta3}(i) we have $h=g_1$, and the linking of the red sphere with the band is opposite to the one $\sref^\circlearrowright_\circlearrowright$ given in Figure~\ref{fig:sref2}(ii). Thus $[f_2]=-\sref^\circlearrowright_\circlearrowright(g_1)$.
    
    Secondly, for the family $f_1$ in Figure~\ref{fig:theta3}(ii) the green guiding arc follows edges $e_1^{-1}e_2$ (giving $g_1$), then $e_2^{-1}e_3^{-1}$ (giving $g_2$), then goes around the meridian to the base sphere (giving $t$), and finally back along $e_3e_2$ (giving $g_2^{-1}$). Thus, in this case $h=g_1g_2tg_2^{-1}$. The red sphere now links correctly, so $[f_1]=\sref^\circlearrowright_\circlearrowright(g_1g_2tg_2^{-1})$, as claimed.
\end{proof}

    We can now prove the first half of Theorem~\ref{thm-intro:W-BG}.
\begin{cor}\label{cor:Theta}
    For a 4-manifold $\M$ and $\Theta_{g_1,g_2}\colon\Theta\hra \M$ there is an equality
    \[
        \wat(\Theta_{g_1,g_2})
        =\ps\circ\sref^\circlearrowright_\circlearrowright(g_1g_2^{-1}tg_2)
        =\ps\circ\prealmap(g_1g_2^{-1}tg_2+(g_1g_2^{-1}tg_2)^{-1}).
    \]
\end{cor}
\begin{proof}
    By Theorem~\ref{thm:Theta} we have $\wat(\Theta_{g_1,g_2})=\ps(\sref^\circlearrowright_\circlearrowright(g_1g_2^{-1}tg_2)-\sref^\circlearrowright_\circlearrowright(g_1))$ 
    and by Theorem~\ref{thm:selfref} we have $\sref^\circlearrowright_\circlearrowright(h)=\prealmap(h+h^{-1})$. Moreover, by Lemma~\ref{lem:g-in-kerr} we have $\prealmap(g_1)=\prealmap(g_1^{-1})=0$, since we are in the case $\X=\M\#\S^1\tm\S^3$. This concludes the proof.
\end{proof}

    Combining Theorem~\ref{thm:Theta} and Corollary~\ref{cor:selfref} gives the following.
\begin{cor}\label{cor:Theta-S4}
    In $\M=\S^4$ there is a unique isotopy class $\Theta\colon\Theta\hra \M$, and we have
    \[
        \wat(\Theta)=\ps\circ
        \sref^\circlearrowright_\circlearrowright(t)
        =
        \ps\circ\prealmap(t).
    \]
\end{cor}

\begin{rem}
    Watanabe's construction is a generalisation of Gusarov--Habiro clasper surgery for classical knots~\cite{Gusarov,Habiro}, and gives more generally classes
\[
    \wat(\Pi_e)\in\pi_{n(d-3)}\Diffp(\M), 
\]
    for an embedding $\Pi_e\colon\Pi\hra \M$ of a trivalent graph $\Pi$ with $2n$ vertices into a smooth $d$-manifold $\M$. This first appeared in \cite{Watanabe-dim-odd} for $d$ odd, then in \cite{Watanabe-dim-4,Watanabe-theta} for $d=4$, and in \cite{Watanabe-dim-4-addendum} for $d$ even. Note that $\Pi=\Theta$ is the single trivalent graph with two vertices.

    Botvinnik and Watanabe showed in \cite{Botvinnik-Watanabe} that in all cases the construction can be simplified so that only a link of two components is used. But this means that the classes are in the image of $\ps\colon \pi_{n(d-3)}\Emb(\nu\S^1,\M\#\,\S^{d-1}\times\S^1)\to \pi_{n(d-3)-1}\Diffp(\M)$. In \cite{K-families} we defined for a $d$-manifold $\X$ graspers $\TG\colon\D^d\hra \X$ of degree $n\geq1$ and used them to construct classes in $\pi_{n(d-3)}\Emb(\S^1,\X)$, analogous to the map $\prealmap$ (which is the case $n=1$ and $d=4$). We will study connections between $\ps\circ\prealmap$ and $\wat(\Pi_e)$ in future work.
\end{rem}

\section{Budney--Gabai's barbell diffeomorphisms}
\label{sec:BG}

\subsection{Definition of a barbell}
\label{subsec:barbell-def}
Fix a neat embedding $u=(u_1,u_2)\colon\D^1\sqcup\D^1\hra\D^4$. Let $bg\in\pi_1(\Embp(\nu\D^1,\D^4\sm\nu u_2);\nu u_1)$ be the loop of framed arcs obtained by twirling a piece of $u_1$ around the meridian sphere $S_2$ of the arc $u_2$, and thickening; see \cite[Rem.\ 5.4]{Budney-Gabai}.

Similarity of $bg$ and the simple grasper family from Figure~\ref{fig:bg-bw}(i), and Section~\ref{sec:graspers}; framings on grasper classes were discussed in Section~\ref{subsec:framed}.
\begin{figure}[!htbp]
    \centering
    \includegraphics[width=0.9\linewidth]{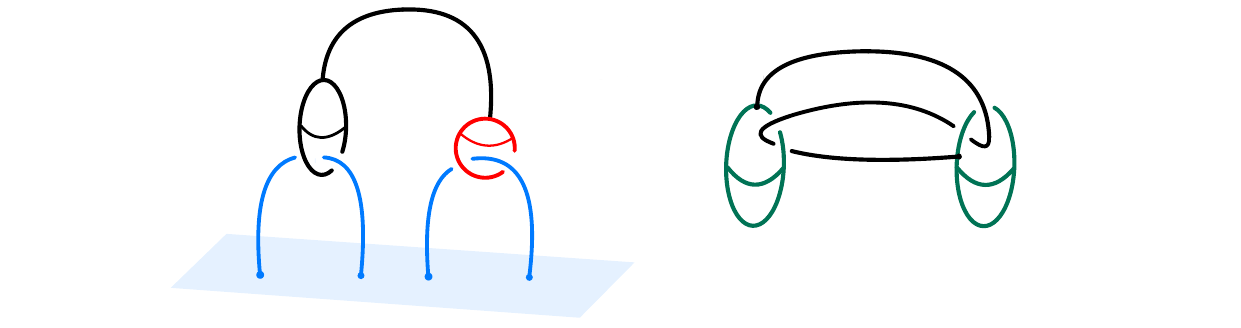}
    \caption{
        (i)~The class $bg\in\pi_1(\Embp(\nu\D^1,\D^4\sm\nu u_2);\nu u_1)$ is obtained by twirling a piece of $u_1$ around the meridian of $u_2$. 
        (ii)~A barbell $\barbemb\colon \barbell\hra \S^4$.
    }
    \label{fig:bg-bw}
\end{figure}
\begin{rem}\label{rem:symmetric}
    Analogously to Lemma~\ref{lem:symmetric}, the family $bg$ can be represented by twirling $u'$ around the negative meridian of $u$, or in a symmetric manner as in Figure~\ref{fig:realmap-symmetric}. 
\end{rem}
\begin{rem}\label{rem:TG*}
    Let $\TG$ be a general grasper with a simple root ball; recall from Definition~\ref{def:root-leaf} that this means that there is a single point $p\in\S^1$ so that $\wh{R}\cap c=\{c(p)\}$. Observe that $\nu(\wh{R}\sqcup L\cup arc)\cong\S^2\tm\D^2$, so $\TG\sm(\nu\wh{R}\cap c)$ is a barbell! We denote it by $\barbemb_\TG$.
    
    Moreover, the grasper family for $\TG$ can be defined as the image under the map
    \[
    \TG_*\colon\pi_1(\Emb(\D^1,\S^2\tm\D^2); u)\ra \pi_1(\Emb(\S^1,\X); c)
    \]
    of the twirling family $bg$ of the arc $u$ around the sphere $\S^2\tm\{0\}$. That is, $\TG_*(\gamma)$ is given as $\TG\circ\gamma\circ\TG^{-1}$ on a neighbourhood of $p\in\S^1$, and by the constant family on the rest of $\S^1$.
\end{rem}

Budney and Gabai consider the manifold $\barbell\coloneqq\D^4\sm \nu(u_1\sqcup u_2)$, called the \emph{model barbell}, and define in \cite[Sec.\ 5]{Budney-Gabai} a diffeomorphism of $\barbell$ rel boundary, called the \emph{barbell map} $\amb(bg)$. It is the result of applying isotopy extension 
\[
    \amb\colon\pi_1(\Embp(\nu\D^1,\D^4\sm\nu u_2);\nu u_1)\ra
    \pi_0 \Diffp(\barbell)
\]
to the loop $bg$.
Namely, consider the fibration sequence $\Diffp(\barbell) \hra \Diffp(\D^4)\to \Embp(\nu(\D^1\sqcup\D^1),\D^4)$ and let $\amb$ be its connecting map (compare with the map $\delta_{\nu c}$ in~\eqref{eq-def:ps} and Remark~\ref{rem:amb}).

Now note that there is an identification with the boundary connect-sum:
    \begin{equation}\label{eq:barbell}
        \S^2\times\D^2\,\natural\,\S^2\times\D^2\cong \D^4\sm \nu(u_1\sqcup u_2)\eqqcolon\barbell,
    \end{equation}
which takes each $\S^2\times\{pt\}$ to the meridian sphere $S_i$ to $u_i$, and each $\{pt\}\times\D^2$ to a 2-disk in $\D^4$ whose boundary $c_i$ has one half of the boundary on $\partial\D^4$ while the other half is a pushoff of $u_i$. 


We define the \emph{spine} of the model barbell to be the the union of the two spheres $S_i$ and the bar, an arc in $\barbell$ connecting them. As in \cite{Budney-Gabai}, we draw an embedding $\barbemb$ by drawing the image of the barbell spine, as in Figure~\ref{fig:bg-bw}(ii) for $\M=\S^4$. In that example the bar describes the word $\bc_2\bc_1$, where $\bc_i$ is the homotopy class of the meridian circle of $\barbemb(S_i)$; compare to Definition~\ref{def:bar-word-and-element}.

\begin{defn}
    Given a smooth 4-manifold $\M$, we call an embedding $\barbemb\colon \barbell\hra \M$ a \emph{barbell}. We define the \emph{barbell diffeomorphism} as
\[
    \bg({\barbemb})\in\pi_0\Diffp(\M)
\]
   given as the barbell map $\amb(bg)$ on the image of $\barbemb$ and the identity on the complement $\M\sm \barbemb(\barbell)$.
\end{defn}

A consequence of Remark~\ref{rem:symmetric} is the following result from~\cite{Budney-Gabai}. See also Lemma~\ref{lem:2-torsion} below.

\begin{lem}[Symmetry of the barbell diffeomorphism]\label{lem:symmetric-barbell}
    If $\barbemb$ is a barbell with both spheres nullhomotopic, then the inverse of $\bg(\barbemb)$ is the barbell diffeomorphism on the same barbell but with exchanged roles of $S_1$ and $S_2$.
\end{lem}

\subsection{Relation to parameterised surgery and graspers}
\label{subsec:bX-ps}
Barbells are related to graspers as we have seen in Remark~\ref{rem:TG*}. Moreover, $\amb$ is defined via an ambient isotopy extension, so one expects a relation between barbell diffeomorphisms and parameterised surgery map from \eqref{eq-def:ps-S}:
\[
\begin{tikzcd}[column sep=1.8cm]
    \ps_{\nu S_1}\colon \pi_1(\Emb(\nu\S^1,\M_{\nu S_1});\nu c) \rar{\delta_{\nu c}} & \pi_0\Diffp(\M\sm\nu S_1) \rar{-\cup\Id_{\nu S_1}} & \pi_0\Diffp(\M),
\end{tikzcd}
\]
where $\M_{\nu S_1}\coloneqq (M\sm\nu S_1)\cup\nu c$.
This is indeed the case as we show now; see also~\cite[Lem.9]{FGHK}.

\begin{lem}\label{lem:grasper-barbell}
    If $\TG$ is a general grasper in $M_{\nu S_1}$ on $\nu c$, such that the root ball is simple, then the diffeomorphism $\delta_{\nu c}(\prealmap^{\TG})$ agrees with the barbell diffeomorphism $\bg(\barbemb_\TG)$ for the barbell $\barbemb_{\TG}$ obtained from $\TG$ by removing the intersection of its root ball with $\nu c$.
\end{lem}
\begin{proof}
    The idea of the proof is that ambient isotopy extension can be exchanged with an embedding. 
    On one hand, $\bg(\barbemb_\TG)$ is by definition the image of $bg$ under the map
    \[\begin{tikzcd}
        \pi_1(\Embp(\nu\D^1,\D^4\sm\nu u_2);\nu u_1)\rar{\amb} & 
        \pi_0 \Diffp(\barbell)\ar{rrr}{(\barbemb_\TG)_*\cup\Id_{(\M\sm\nu S_1)\sm\barbemb_\TG(\barbell)}} &&&
        \pi_0\Diffp(\M\sm\nu S_1).
    \end{tikzcd}
    \]
    On the other hand, by Remark~\ref{rem:TG*} we have $\prealmap^{\TG_\barbemb}=\TG_*(bg)$, so we can consider the square
    \[\begin{tikzcd}
    \pi_1(\Embp(\nu\D^1,\D^4\sm\nu u_2);\nu u_1)\rar{\amb}\dar{\TG_*} &
    \pi_0 \Diffp(\barbell)\dar{(\barbemb_\TG)_*\cup\Id_{(\M\sm\nu S_1)\sm\barbemb_\TG(\barbell)}}\\
    \pi_1(\Emb(\nu\S^1,\M_{\nu S_1});\nu c) \rar{\delta_{\nu c}} 
    & \pi_0\Diffp(\M\sm\nu S_1) 
    \end{tikzcd}
\]
    This commutes because both $\amb$ and $\delta_{\nu c}$ are defined by ambient isotopy extension, and this depends on integration of vector fields: this is done locally, so can be performed in $\M\sm\nu S_1$ by embedding the result of doing it in $\barbell$. This proves the desired result.
\end{proof}

\begin{thm}\label{thm:BG}
    For any 4-manifold $\M$ and an embedding $\barbemb\colon \barbell\hra \M$ we have
    \[
        \bg({\barbemb}) = \ps_{\nu \barbemb S_1}(\prealmap^{\TG_\barbemb})
    \]
    for the following general grasper $\TG_\barbemb\subset\M_{\nu\barbemb S_1}$ on $\nu c$. Its simple root ball $\nu\wh{R}$ is the union of $\nu c|_{\nu p}$ and $\nu S_1$, its leaf is $S_2\subseteq\M_{\nu\barbemb S_1}$, and its bar is the same as the bar for $\barbemb$ in $\M$.
\end{thm}

\begin{proof}
    The map $\ps_{\nu \barbemb S_1}(\prealmap^{\TG_\barbemb})$ is obtained from $\delta_{\nu c}(\prealmap^{\TG_\barbemb})$ by extending by the identity over $\nu S_1$. By Lemma~\ref{lem:grasper-barbell} we have $\delta_{\nu c}(\prealmap^{\TG_\barbemb})=\bg(\barbemb_{\TG_\barbemb})$ in $\M\sm\nu S_1$ for the barbell $\barbemb_{\TG_\barbemb}$ obtained from $\TG_\barbemb$ by removing the intersection of its root ball with $\nu c$. Thus, one of the cuffs of $\barbemb_{\TG_\barbemb}$ is $S_2$ and the other one is precisely $S_1$: when we remove $\nu\wh{R}\cap\nu c$ what remains is $\nu S_1$.
    Therefore, once we put $\nu S_1$ back this barbell is the same as $\barbemb$.
\end{proof}

If $\barbemb S_2$ is nullhomotopic in $\M\sm\nu\barbemb S_1$, it is in $X\coloneqq \M_{\nu\barbemb S_1}$ as well, so the corresponding $\TG_\barbemb$ is an example of a simple-null grasper in $\X$, as studied in Section~\ref{subsec:simple-null-graspers}.
The following is a consequence of Theorems~\ref{thm:simple-null-graspers} and~\ref{thm:BG}.
\begin{cor}\label{cor:implant-formula}
    Let $\M$ be any 4-manifold and $\barbemb\colon \barbell\hra \M$ a barbell with $L\coloneqq\barbemb S_2$ nullhomotopic in $\M\sm\barbemb S_1$. Let $\W\in\pi_1(\M\sm\nu\barbemb(S_1\sqcup S_2))$ be the bar word.
    Then
    \[
        \bg({\barbemb}) 
        = \ps_{\nu \barbemb S_1}\circ\prealmap^{\TG_{\W L}} 
        = \ps_{\nu \barbemb S_1}\circ\prealmap\circ(i_L)_*
    \Big(
        \W\dax(\wh{\wh{L}})\W^{-1}
    -\lambdabar(\W\wh{\wh{L}},\W)
    -\ol{\lambdabar(\W\wh{\wh{L}},\W)}
    \Big).
    \]
    for any ball $\wh{L}$, $\partial\wh{L}=L$.
    Moreover, if $\wh{L}$ is embedded (i.e.\ $L$ is unknotted), the first term vanishes.
\end{cor}

\section{Examples}
\label{sec:examples}

\subsection{Half-unknotted barbells}
\label{subsec:barbells-half-unknotted}

Let us now assume a barbell has the cuff $L=\barbemb S_2$ unknotted, that is, $L$ bounds an embedded 3-ball $\wh{L}$ in $\M$. Let us write $\X\coloneqq \M_{\nu\barbemb S_1}$ and $\X_{\nu L}\coloneqq \M_{\nu\barbemb (S_1\sqcup S_2)}$. We have
\begin{align*}
    \pi_1\X & \cong \pi_1(\M\sm\nu\barbemb S_1),\\
    \pi_1\X_{\nu L} & \cong\pi_1(\X\sm\nu L) \cong\pi_1(\M\sm\nu\barbemb(S_1\sqcup S_2))\cong\pi_1(\M\sm\nu\barbemb S_1)\ast\Z.
\end{align*}
The generator of the free factor in $\pi_1(\X\sm\nu L)$ is $\by=[\barbemb(c_2)]$, and the map $i_L\colon\pi_1(\X\sm\nu L)\sra\pi_1\X$ sends it to $1$. Recall the bar word $\W$ and bar group element $\bw=i_L(\W)$ from Definition~\ref{def:bar-word-and-element}.

\begin{prop}\label{prop:half-unknotted}
    Let $\M$ be any 4-manifold and $\barbemb\colon \barbell\hra \M$ a barbell with unknotted cuff $L=\barbemb S_2$. 
    Write the bar word as
    \[
    \W
    =\prod_{i=1}^r
    f_i\by^{\e_i}h_i\in\pi_1(\M\sm\nu\barbemb S_1)\ast\Z,
    \]
    for $f_i,h_i\in\pi_1(\M\sm\nu\barbemb S_1)$ and $\e_i\in\{-1,1\}$. Denote ${\bw}_i=\prod_{j=1}^i
    f_jh_j$, and note that $\bw_r=\bw$ is the bar group element.
    Then the corresponding barbell diffeomorphism $\bg(\barbemb)\in\pi_0\Diffp(\M)$ satisfies
    \[
        \bg(\barbemb)=\ps_{\nu \barbemb S_1}\circ
        \sref^\circlearrowright
        \Big(
            \bw\sum_{i=1}^r\e_if_i^{-1}{\bw}_{i-1}^{-1}
        \Big).
    \]
\end{prop}
\begin{proof}
    We can pick for $\wh{L}\colon\D^3\hra \M$ the standard 3-ball bounded by the unknot $L=\barbemb S_2$. Then by Corollary~\ref{cor:implant-formula} we have
       \[
        \bg({\barbemb}) 
        = \ps_{\nu \barbemb S_1}\circ\prealmap^{\TG_{\W L}} 
        = \ps_{\nu \barbemb S_1}\circ\prealmap\circ (i_L)_*
    \Big(
    - \lambdabar(\W\wh{\wh{L}},\W)
    -\ol{ \lambdabar(\W\wh{\wh{L}},\W)}
    \Big).
    \]
    If we let $h= (i_L)_*(\lambdabar(\W\wh{\wh{L}},\W))$, then using Theorem~\ref{thm:selfref} we have
    \[
        \prealmap^{\TG_{\W L}}=\prealmap(- h- h^{-1})=\sref^\circlearrowright(- h),
    \]
    so we need only show that $h=\bw\sum_{i=1}^r-\e_if_i^{-1}{\bw}_{i-1}^{-1}$.

    Indeed, the intersections of $\W \wh{L}$ with $\W$ correspond to the occurrences of $\by^{\pm1}$ in the word $\W$, and the associated group element for $\by^{\e_i}$, for $i=1,\dots,r$, first goes along $\W$ (on $\W \wh{L}$), and then along the inverse of $(\prod_{j=1}^{i-1}
    f_j\by^{\e_j}h_j)f_i$ (on $\W$). Applying $(i_L)_*$ to this we obtain
    \[
    (i_L)_*(\W f_i^{-1}(\prod_{j=1}^{i-1}
    f_j\by^{\e_j}h_j)^{-1})
    =\bw f_i^{-1}(\prod_{j=1}^{i-1}f_jh_j)^{-1}
    =\bw f_i^{-1}\bw_{i-1}^{-1}.
    \]
    Finally, to see that the sign is precisely opposite to $\e_i$ it suffices to show $\lambda(\wh{L},\by)=-1$. Recall that by our orientation convention the circle $\by$ is positively oriented if the orientation  (tangent 2-space of $L$, normal 2-disk to $L$) is positive, whereas the sign for $\lambda(\wh{L},\by)$ is computed using (tangent 3-space of $\wh{L}$, tangent 1-space of $\by$). Since the outward normals of $\by$ and $L$ are parallel but oriented oppositely, we get the claimed sign. 
\end{proof}

\subsection{Unknotted barbells}
\label{subsec:barbells-unknotted}
Let us now assume that a barbell $\barbemb\colon\barbell\hra \M$ has unknotted cuff $\barbemb S_1$. Then $\pi_1(\M\sm\nu\barbemb S_1)\cong\pi_1\M\ast\Z$ is generated by $\pi_1\M$ and $\bx=t\in\Z$. Moreover, as in~\eqref{eq-def:ps} we write
\[
    \ps\coloneqq\ps_{\nu\barbemb S_1}\colon\pi_1(\Emb(\nu\S^1,\M\#\,\S^3\times\S^1);\nu c_1)\to\pi_0\Diff(\M),
\]


\begin{example}\label{ex:bar-element}
    Assume $\barbemb\colon \barbell\hra \M$ is a barbell in a 4-manifold $\M$ that has both cuffs unknotted. Fix some $h\in\pi_1\M\ast\Z$. The following computations are immediate from Proposition~\ref{prop:half-unknotted}.
    \begin{enumerate}
        \item 
            $\bg(\barbemb_{\by h})=\ps\circ\sref^\circlearrowright(h)$ as  $\bw=\bw_1=h$ and $f_1=1$.
            In particular, for $g_1,g_2\in\pi_1M$ we have
            \[
            \bg(\barbemb_{\by (g_1g_2^{-1} \bx g_2)})
            =\ps\circ\sref^\circlearrowright(g_1g_2^{-1}tg_2).
        \]
        \item 
            $\bg(\barbemb_{(\by^{\e_1}h_1)(\by^{\e_2}h_2)})=\ps\circ
        \sref^\circlearrowright(\e_1 h_1h_2+ \e_2 h_1h_2h_1^{-1})$ 
            as $\bw=\bw_2=h_1h_2$ and $\bw_1=h_1$.
        \item 
            $\bg(\barbemb_{(f_1\by^{\e_1}h_1)(\by^{\e_2}h_2)})=\ps\circ
        \sref^\circlearrowright(\e_1 f_1h_1h_2f_1^{-1} + \e_2 f_1h_1h_2(f_1h_1)^{-1})$ 
            as $\bw=\bw_2=f_1h_1h_2$ and $\bw_1=f_1h_1$.
            In particular, 
            \[
            \bg(\barbemb_{(g_2^{-1}\by^{-1}\bx^{-1})(\by \bx g_2g_1g_2^{-1})})
            =\ps\circ\sref^\circlearrowright(-g_1g_2^{-1}g_2+ g_1g_2^{-1}(g_2^{-1} t^{-1})^{-1})
            = \ps\circ\sref^\circlearrowright 
        (g_1g_2^{-1}tg_2-g_1)
        \]
            as $f_1h_1h_2=g_2^{-1}t^{-1}tg_2g_1g_2^{-1}=g_1g_2^{-1}$  and $f_1=g_2^{-1}$ and $f_1h_1=g_2^{-1} t^{-1}$.\qedhere
    \end{enumerate}
\end{example}

We use the last example to directly relate Watanabe's theta classes and Budney--Gabai's barbell diffeomorphisms, completing the proof of Theorem~\ref{thm-intro:W-BG}.
\begin{cor}\label{cor:Wat-implant}
    If a barbell $\barbemb\colon\barbell\hra \M$ in a 4-manifold $\M$ has both cuffs unknotted and the bar word $\W=\by (g_1g_2^{-1} \bx g_2)$, then
    \[
        \bg(\barbemb_{\W})=\wat(\Theta_{g_1,g_2}).
    \]
\end{cor}
\begin{proof}
    On one hand, recall from Corollary~\ref{cor:Theta} that $\wat(\Theta_{g_1,g_2})=\ps\circ\prealmap(h+h^{-1})$
    for $h=g_1g_2^{-1}tg_2$. On the other hand, in Example~\ref{ex:bar-element}(1) we have computed $\bg(\barbemb_{\W})=\ps\circ\sref^\circlearrowright(h)$ for precisely the same $h$.
    Finally, by Theorem~\ref{thm:selfref} these families agree: $\sref^\circlearrowright(h)=\prealmap(h+h^{-1})$.
\end{proof}

Finally, let us relate the diffeomorphisms of $\M=\D^3\times\S^1$ from \cite[Constr.\ 6.11]{Budney-Gabai} to grasper classes; the following was stated in the introduction as Corollary~\ref{cor-intro:BG}.

\begin{cor}\label{cor:implant-D3xS1}
    Consider $\M=\D^3\times\S^1$ and the barbells $\delta_m$, $m\geq4$, as in Figure~\ref{fig-intro:barbell}(ii). Let $g$ denote the generator of $\pi_1\M\cong\Z$. Then we have
    $\W=g\by(g^{m-3}\bx g^2)$ and $\bw=g^{m-2}tg^2$ and
    \[
        \bg(\delta_m)=
        \ps\circ\sref^\circlearrowright(g^{m-2}tg)
        =\ps\circ\prealmap(g^{m-2}tg+\ol{g^{m-2}tg}).
    \]
\end{cor}
\begin{proof}
    We have $f_1=g$, $h_1=g^{m-3}tg^2$, so $\bw f_1^{-1}=g^{m-2}tg^2g^{-1}=g^{m-2}tg$. Thus, Proposition~\ref{prop:half-unknotted} implies the first equality. 
    For the second we use  Theorem~\ref{thm:selfref}.
\end{proof}

\subsection{Barbells in the 4-sphere}
We apply the results of the previous section to the case $\M=\S^4$.

\begin{cor}\label{cor:implant-D4}
    Consider $\M=\S^4$ and an embedding $\barbemb_{\W}\colon\barbell\hra \S^4$ with both cuffs unknotted, and the bar word $\W=\prod_{i=1}^r\by^{\e_i}\bx^{n_i}\in\pi_1(\S^4\sm\nu\barbemb(S_1\sqcup S_2))\cong\Z\ast\Z$. Then the corresponding barbell diffeomorphism $\bg(\barbemb_{\W})\in\pi_0\Diff(\S^4)$ satisfies
    \[
        \bg(\barbemb_{\W})=\ps\circ
        \sref^\circlearrowright
        \Big(
            \sum_{i=1}^r\e_it^{n_i+\dots +n_r}
        \Big).
    \]
\end{cor}
\begin{proof}
        We simply put $f_i=1$ and $h_i=\bx^{n_i}$ into the formula of Proposition~\ref{prop:half-unknotted}, so that $\bw=t^{n_1+\dots+n_r}$ and $\e_if_i^{-1}{\bw}_{i-1}^{-1}=\e_it^{-n_1-\dots-n_{i-1}}$.
\end{proof}


This means that for $\M=\S^4$ it suffices to consider the simplest bar words $\bw=\by\bx^i$, giving 
\[
    \bg(\barbemb_{\by\bx^i})=\ps\circ\sref^\circlearrowright(t^i).
\]
By \eqref{eq:case-S3xS1} we can assume $i\geq1$. Note that $\bg(\barbemb_{\by\bx^i})^{-1}=\bg(\barbemb_{\by^{-1}\bx^i})$ (since this is $\ps\circ\sref^\circlearrowright(-t^i)$).

\begin{example}\label{cor:implant-D4-vx}
    For example, Budney and Gabai conjecture \cite[Conj.\ 5.18]{Budney-Gabai} that for the word $\W=\by\bx$ and $\W=\by\bx^{-1}$ the barbell diffeomorphism is nontrivial. Combining Corollaries~\ref{cor:implant-D4} and ~\ref{cor:selfref} to these cases we have
    \begin{align*}
        \bg(\barbemb_{\by\bx})&=\ps\circ\sref^\circlearrowright(t)=\ps\circ\prealmap(t),\\
        \bg(\barbemb_{\by\bx^{-1}})&=\ps\circ\sref^\circlearrowright(t^{-1})=\ps\circ\prealmap(t).
    \end{align*}
    In particular, the barbell diffeomorphisms $\bg(\barbemb_{\by\bx^{-1}})$ and $\bg(\barbemb_{\by\bx})$ are isotopic.
\end{example}

In fact, Budney and Gabai show in~\cite[Prop.\ 5.17]{Budney-Gabai} that the last example has order two, as well as any barbell diffeomorphism of $\S^4$ with unknotted cuffs and the bar word $w\in\Z\ast\Z$ which is inverse-palindromic, that is, $\ol{w^{-1}}=w$, where the involution $\ol{\cdot}\colon\Z\ast\Z\to\Z\ast\Z$ exchanges the two generators.
The proof is based on the symmetry of the barbell diffeomorphism (see Lemma~\ref{lem:symmetric-barbell}), also used by Gay and Hartman in a similar result~\cite[Lem.\ 10]{Gay-Hartman}. We recast this argument in our language as follows.
    
\begin{lem}\label{lem:2-torsion}
    We have
    \[
        \bg(\barbemb_{\by\bx^i})^{-1}=\bg(\barbemb_{\by\bx})^i.
\]
    In particular, $\bg(\barbemb_{\by\bx})^2=\Id$.
\end{lem}
Recall that $\bg(\barbemb_{\by\bx})=\wat(\Theta)=\ps\circ\prealmap(t)$, so this completes the proof of Corollary~\ref{cor-intro:final-S4}.

\begin{proof}
    By Remark~\ref{rem:symmetric} we can obtain $\bg(\barbemb_{\by\bx^i})^{-1}$ by switching the choice of spheres in the barbell: we twirl a meridian $c_2$ around $S_1$, following the same bar. We now read the bar word backwards and with $\bx$ and $\by$ exchanged, so $\bg(\barbemb_{\by\bx^i})^{-1}=\bg(\barbemb_{\by^i\bx})$.
    By Corollary~\ref{cor:implant-D4} we then have 
\begin{align*}
    \bg(\barbemb_{\by\bx^i})^{-1}
    =\bg(\barbemb_{\by^i\bx})&=\ps\circ\sref^\circlearrowright(it)
    =\ps\circ\prealmap(it+it^{-1})=(\ps\circ\prealmap(t+t^{-1}))^i=(\bg(\barbemb_{\by\bx}))^i.
\end{align*}
For the penultimate equality we use that $\ps\circ\prealmap$ is a group homomorphism.
\end{proof}


   



\printbibliography[heading=bibintoc]

\vspace{10pt}
\hrule

\end{document}